\documentclass[11pt,reqno,a4paper]{amsart}
\usepackage[margin=1in]{geometry}
\usepackage{amsmath, amsthm, amssymb, mathrsfs, mathdots}
\usepackage{enumitem}
\usepackage[colorlinks=true, pdfstartview=FitV, linkcolor=blue,citecolor=blue, urlcolor=blue]{hyperref}
\usepackage{times}
\usepackage[dvipsnames]{xcolor}
\usepackage{subcaption}
\usepackage{mathtools}
\usepackage{bbm}
\usepackage{upgreek}
\usepackage{comment}
\usepackage{dsfont}
\usepackage{tikz}
\usetikzlibrary{calc, decorations.markings, shapes, arrows.meta, patterns, positioning}
\usepackage{graphicx}

 \usepackage[capitalise]{cleveref}

\crefformat{equation}{(#2#1#3)}
\crefrangeformat{equation}{(#3#1#4) to~(#5#2#6)}
\crefmultiformat{equation}{(#2#1#3)}%
{ and~(#2#1#3)}{, (#2#1#3)}{ and~(#2#1#3)}

\def\epsilon{\varepsilon}
 
\def\de{{\partial}}

\def\RR {\mathbb{R}}
\def\CC {\mathbb{C}}
\def\TT {\mathbb{T}}
\def\ZZ {\mathbb{Z}}
\def\NN {\mathbb{N}}

\def\Re{{\rm Re}}
\def\Im{{\rm Im}}
\def\uno{{\rm Id}}

\def\Log{{\rm Log}}

\newcommand{\cH}{\mathcal{H}}
\newcommand{\cI}{\mathcal{I}}
\newcommand{\cJ}{\mathcal{J}}
\newcommand{\cK}{\mathcal{K}}
\newcommand{\cL}{\mathcal{L}}
\newcommand{\cM}{\mathcal{M}}

\newcommand{\cP}{\mathcal{P}}

\newcommand{\cR}{\mathcal{R}}
\newcommand{\cS}{\mathcal{S}}
\newcommand{\cT}{\mathcal{T}}

\newcommand{\cW}{\mathcal{W}}

\newcommand{\Ker}{\mathrm{Ker}}
\newcommand{\Rn}{\mathrm{Ran}}

\newcommand{\mrA}{\mathrm{A}}
\newcommand{\mrB}{\mathrm{B}}
\newcommand{\mrC}{\mathrm{C}}
\newcommand{\im}{\mathrm{i}\,}

\newtheorem{proposition}{Proposition}[section]
\newtheorem{theorem}[proposition]{Theorem}
\newtheorem{corollary}[proposition]{Corollary}
\newtheorem{lemma}[proposition]{Lemma}

\theoremstyle{definition}
\newtheorem{definition}[proposition]{Definition}
\newtheorem{remark}[proposition]{Remark}

\theoremstyle{definition}

\numberwithin{equation}{section}

\title{Instability of  two-dimensional Taylor--Green Vortices}

\author{Gonzalo Cao-Labora}
\author{Maria Colombo}
\author{Michele Dolce}
\author{Paolo Ventura}
\address{Institute of Mathematics, EPFL, Station 8, 1015 Lausanne, Switzerland}
\email{gonzalo.caolabora@epfl.ch}
\email{maria.colombo@epfl.ch}
\email{michele.dolce@epfl.ch}
\email{paolo.ventura@epfl.ch}

\begin{document}

\begin{abstract}
For a wide class of linear Hamiltonian operators we develop a general criterion that characterizes the unstable eigenvalues as the zeros of a holomorphic function
given by the determinant of a finite-dimensional matrix. We apply the latter result to prove the spectral instability of the Taylor--Green vortex in two-dimensional ideal fluids. The linearized Euler operator at this steady state possesses different invariant subspaces, within which we apply our criterion to rule out or detect instabilities.
We show linear stability of odd perturbations, for which the unstable spectrum can appear only on the real axis. We exclude this possibility by applying 
 our stability criterion. 
 Real instabilities, instead, exist and can be detected with the same criterion if we consider suitable rescalings of the Taylor--Green vortex.
 In the subspace of functions even in both variables, the problem is reduced to finding a single complex root of our stability function. We successfully locate this value by combining our general criterion with a rigorous computer-assisted argument. As a consequence, we fully characterize the unstable spectrum of the Taylor--Green vortex. 
\end{abstract}

\maketitle

\section{Introduction}

The motion of a two-dimensional {(2D)} incompressible ideal fluid, evolving in a periodic box, is described by the classical Euler equations. In vorticity formulation they are written as
\begin{equation}
    \label{Euler2Dvorticity}
\begin{cases}
    \partial_t \omega +u\cdot  \nabla  \omega = 0\qquad \text{with } (x,y)\in \TT^2,\, t\geq 0,\\
    u=\nabla^\perp\psi, \qquad \Delta\psi =\omega,
\end{cases}
\end{equation}
where $\TT^2=\RR^2/(2\pi \ZZ^2)$, $u$ is the velocity field,  $\omega=\nabla^\perp\cdot v$ is the vorticity with $\nabla^\perp=(-\de_y,\de_{{x}})$. 

In this paper, we aim at studying the spectral stability properties of the 2--dimensional \emph{Taylor--Green vortex} (also known as \emph{cellular flow}).
This is  a steady solution of \eqref{Euler2Dvorticity}
whose vorticity, stream function and velocity field can be respectively given by
\begin{equation}\label{TaylorGreen11}
\omega_E(x,y):= 2\sin(x)\sin(y), \qquad \psi_E(x,y):=-\sin(x)\sin(y)\, , \qquad  U_E\coloneqq\nabla^\perp \psi_E,
\end{equation}
and note that $\omega_E=F(\Psi_E)$ with $F(s)=-2s$. This flow is commonly used as a benchmark in 2D computational fluid dynamics \cite{sengupta2018non} and was first studied within a wider class of $3D$ vortices by Taylor and Green in \cite{taylor1937mechanism}. Particular features of this flow are its symmetries and the presence of hyperbolic stagnation points, structures  that have been used  in a number of results, from the degeneration of the H\"older regularity of the flow map in \cite{bahouri1994equations} to the  vorticity gradient growth results in \cite{Denisov09,denisov2015double,kiselev2014small,zlatovs2015exponential}. On the other hand, the study of spectral stability properties of steady states is only well-understood for unidirectional flows as shear flows or radial vortices. This analysis was developed in an impressive literature, starting from the seminal contributions of Rayleigh \cite{Rayleigh80}, Kelvin \cite{Thomson1880Proc} and Tollmien \cite{Tollmien1935}, we refer to \cite{Arnold65,friedlander1997nonlinear,lin2003instability,vishik2018instability,Grenier16Duke,bedrossian2019vortex,grenier2020linear} and the book \cite{drazin2004hydrodynamic} on 
hydrodynamic stability for further references and applications of such analysis. Outside the realm of unidirectional flows, only a few notable exceptions relying on special structures are partially understood, namely  certain \textit{cat's eye} flows 
\cite{friedlander2000unstable}, via averaging methods around highly oscillating unstable shear flows, and
 the Kelvin-Stuart's vortices 
 \cite{liao2023stability}. 
 
 For general streamlines geometry the spectral problem complicates significantly, and the Taylor--Green vortex is a fundamental example of an analytic steady state in $\mathbb{T}^2$ 
   whose spectral properties have long remained elusive from the mathematical perspective, see for instance the open problem on the Taylor--Green vortex instability in the recent review by Drivas and Elgindi \cite{Drivas23}. Only recently Zhao, Protas and Shvydkoy \cite{zhao2024inviscid} provided numerical evidence that the flow in \eqref{TaylorGreen11} is spectrally unstable.

To state our result, we first observe that for steady states in which there is a \textit{global} functional relationship between stream function and vorticity, it is convenient to introduce the Poisson bracket notation 
\begin{equation}\label{PoissonBracket}
\{f,g\} := \nabla^\perp f \cdot \nabla g = \de_{x} f \de_{y} g - \de_{y} f \de_{x} g\,.
\end{equation}
In this way, considering a perturbation of the Taylor-Green vortex as $\omega=\omega_E+w$ and linearizing the system \eqref{Euler2Dvorticity} at this steady state, one has that \begin{equation}
    \label{eq:perturbation}\de_tw=\cL_E(w),
\end{equation} with
\begin{align}
\label{cL}
    \cL_E(w):=-\big(\{\psi_E,w\}+\big\{\Delta^{-1}w,F(\psi_E)\big\}\big)=-\{\psi_E,(\uno+2\Delta^{-1})w\} \,, 
\end{align}
where $\cL_E:{{\rm Dom}(\cL_E)\subset L^2_0(\TT^2)}\to L^2_0(\TT^2)${, ${\rm Dom}(\cL_E):= \big\{ w \in L^2_0(\TT^2): \{ \psi_E,w\} \in L^2_0(\TT^2) \big\}$,} and the subscript $0$ is used to denote spaces restricted to zero-average functions.

The main result of this paper is the introduction of a novel instability criterion, discussed more precisely later on, that allows us to fully characterize the spectrum of $\cL_E$ outside of the imaginary axis.
\begin{theorem}\label{linthm}
Let $\cL_E$ be the operator defined in \eqref{cL}. There exists a  unique $\lambda_\star\in [0.1,0.17]-\im [0.57,0.63] $ such that the spectrum of $\cL_E$ outside of the imaginary axis consists of four simple eigenvalues given by
    \begin{equation}
        \label{theeigs}
 \sigma(\cL_E) \setminus \im \mathbb{R} = \{ \lambda_\star, \bar \lambda_\star, -\lambda_\star, -\bar \lambda_\star \}  .  \end{equation}
\end{theorem}

 The theorem above fully characterizes the \textit{unstable} spectrum of the Taylor-Green vortex. Moreover, applying Shvydkoy and Latushkin's result in \cite{shvidkoy2003essential,shvydkoy2005essential}, the essential spectrum of $\cL_E$ equals $\im \mathbb{R}$.  More in general, they prove (see also Vishik's work \cite{vishik1996spectrum} for $m=-1$) that if the operator is considered  in $H^m(\TT^2)$ with $m\in \ZZ$, the essential spectrum is 
 $\{ |\Re z| \leq |m\mu|\}$, where $\mu$ the Lyapunov exponent of the steady velocity field $U_E$  (in our case $\mu=1$). 
  Thus, our $\lambda_\star$ is embedded in the essential spectrum for any $m\neq 0$, and in particular the velocity formulation in $L^2(\TT^2)$ corresponds to $m=-1$ in vorticity.  
 
 A long-standing line of research has addressed the problem of promoting linear instability to nonlinear instability. 
 In fluid mechanics, Friedlander, Strauss and Vishik \cite{friedlander1997nonlinear}, Grenier \cite{grenier2000nonlinear}, and Bardos, Guo and Strauss \cite{bardos2002stable}, among many others, established general criteria to infer nonlinear instability from linear instability. Unfortunately, these results require either additional regularity of the unstable eigenfunction \cite{friedlander1997nonlinear,grenier2000nonlinear} or the condition that the real part of the unstable eigenvalue exceed the Lyapunov exponent $\mu$ 
 \cite{bardos2002stable}, which equivalently requires that the eigenvalue is out of the essential spectrum in velocity formulation. 
On the other hand, Lin \cite{Lin04IMRS} proved the most general result that can actually be applied to our context, allowing us to deduce two important properties from Theorem~\ref{linthm}: first, an $L^2$ unstable eigenfunction automatically gains some regularity, with a nontrivial relation between the Lyapunov exponent and the real part of the unstable eigenvalue; second, linear instability implies nonlinear instability in \emph{velocity}. More precisely, as a direct consequence of \cite{Lin04IMRS} and Theorem~\ref{linthm}, we have the following.
\begin{corollary}\label{cor:lin}
The eigenfunction $\omega_\star\in L^2(\TT^2)$ of $\cL_E$ associated with the simple eigenvalue $\lambda_\star$ satisfies 
\begin{equation}\label{eqn:somereg}
    \omega_\star \in (W^{1,p}\cap L^q)(\TT^2)\qquad \mbox{for any }1\le p<1/(1-\Re(\lambda_\star))\mbox{ and }1\le q<\infty.
\end{equation}
Moreover, there exist constants $C_0, \beta_0, \delta_0>0$ such that for any $0<\delta<\delta_0$ the solution $\omega(t),u(t)$  of  \eqref{Euler2Dvorticity} with initial datum $\omega(0,\cdot)=\omega_E+\delta\,\Re(\omega_\star)$ satisfies
\begin{align}\label{eq:nonlininst}
\sup_{0<t<C_0|\ln(\delta)|}\|u(t)-u_{\rm in}\|_{L^{2}}\geq \beta_0.
\end{align}
\end{corollary}

\begin{figure}[htbp]
    \centering
    \begin{subfigure}[b]{0.48\textwidth}
        \centering
        \includegraphics[width=\textwidth]{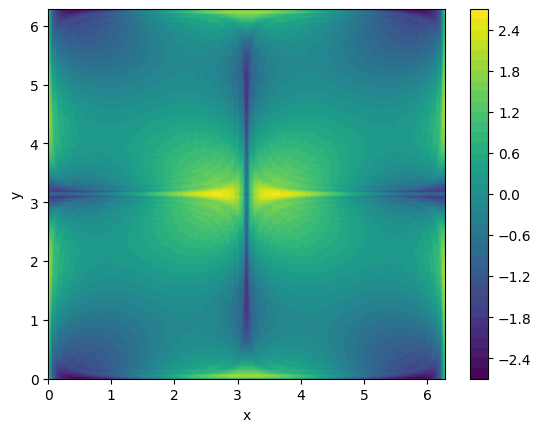}
        \label{fig:eigenfunction}
    \end{subfigure}
    \hfill 
    \begin{subfigure}[b]{0.48\textwidth}
        \centering
        \includegraphics[width=\textwidth]{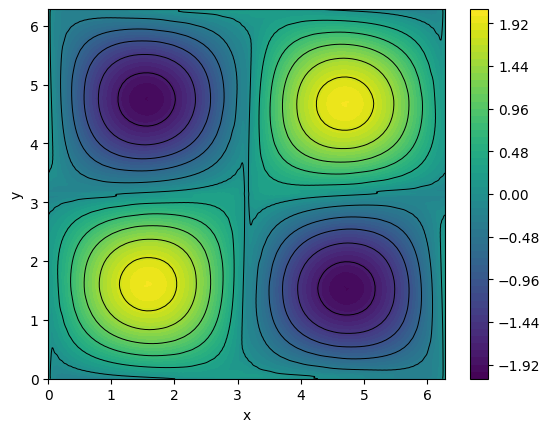}
        \label{fig:instability}
    \end{subfigure}

    \caption{On the left, real part of the (nonsmooth) eigenfunction $\omega_\star$, normalized so that $\| \omega_\star \|_{L^2 ([0, 2\pi]^2)} = \| \omega_E \|_{L^2 ([0, 2\pi]^2)} = 2\pi$. On the right, the real part of the perturbation $\omega_E + \delta \omega_\star$ with $\delta = 0.1$. 
    Such eigenfunction and perturbation break the original symmetries of the cellular flow (odd with respect to both variables). 
    }
    \label{fig:instability_parallel}
\end{figure}

The nonlinear instability in velocity already provides an interesting and physically relevant mechanism, measured here in terms of growth of the kinetic energy of the perturbation. Note that this result does not follow by a standard perturbative argument, because $\lambda_\star$ is embedded in the essential spectrum in velocity formulation. These problems are bypassed in \cite{Lin04IMRS} with a clever scheme that studies the vorticity and velocity formulation as a coupled system. At present, it is also the only mathematically achievable result for upgrading linear instability to nonlinear instability without sufficient regularity of the unstable eigenfunction, around which one would ideally run a perturbative bootstrap argument. In our case, the best regularity we can deduce from Corollary~\ref{cor:lin} yields $\omega_\star \in H^{\frac15}(\TT^2)$, using $\mathrm{Re}(\lambda_\ast) > \frac{1}{10}$ and standard Sobolev embeddings. Using the numerically obtained (but not rigorously justified) value of $\lambda_\star$ would give $s\approx 0.28$. In any case, this amount of regularity seems insufficient to directly tackle nonlinear instability in vorticity with current techniques, and the eigenfunction is not expected to be smooth. More precisely, we do not expect $\omega_\star$ to be $H^1$ at the hyperbolic point, see Figure~\ref{fig:instability_parallel} and Remark~\ref{rmk:nonsmooth} below.

{Finally, we observe that the spectral instability of $\omega_E$ also implies the existence of a {pair of conjugated} unstable eigenvalues for the operator 
\begin{align}
\label{def:Lv}
    \cL_\nu\coloneqq \cL_{E}+\nu\Delta,
\end{align}
when $\nu>0$ is sufficiently small. The operator above arises from the linearization of the 2D forced Navier-Stokes at $\omega_E$, where the force is precisely chosen so that $\omega_E$ is a steady state for the viscous problem (see Section \ref{sec:NS} for more details).  As a consequence of the instability for $\cL_E$ in Theorem \ref{linthm}, we have the following.
\begin{corollary}
\label{cor:NS}
    For $\nu>0$ sufficiently small the unstable spectrum of the operator $\cL_\nu$  in \eqref{def:Lv} contains a pair of eigenvalues
    $
    \{ \lambda_\nu, \bar\lambda_\nu\} \subseteq \sigma_{L^2}(\cL_\nu) \cap \{ \Re(z) > 0\}
    $,
    with $\lambda_\nu\to \lambda_\star$ as ${\nu\to 0^+}$. 
\end{corollary}
To prove this result, one can look at $\cL_{\nu}$ as a singular perturbation of $\cL_E$ since the unstable eigenvalues of $\cL_E$ are  isolated.
Then, arguments involving precise convergence of the resolvents analogous to \cite{albritton2022non} or \cite{shvydkoy2008unstable} could be applied in our case. However, in Section \ref{sec:NS} we will prove Corollary \ref{cor:NS} by adapting our instability criterion 
in the next section. For this argument it will be enough to exploit some weak convergence properties of the resolvent of the operator involved. Finally, combining Corollary \ref{cor:NS} with the results in \cite{friedlander2006nonlinear} one can deduce the nonlinear instability (in velocity) for the Taylor-Green vortex in the 2D forced Navier-Stokes equations.}

\subsection{A new instability criterion for linear Hamiltonian operators} A key novelty of this paper lies in the method of proof for Theorem \ref{linthm}. It is based on a new, general criterion for detecting unstable eigenvalues of \emph{Hamiltonian operators} (as defined below),  whose application to the Taylor-Green vortex yields the result stated in Theorem \ref{linthm}. We expect the criterion to apply to many stationary solutions of Euler, for instance we describe in Section~\ref{sec:nmTG} how we can detect with it many different instabilities of the generalized $(m,n)$ Taylor--Green vortex. 
\begin{definition}[Hamiltonian Operator]
\label{def:Hamiltonian}
Let ${\rm X}$ be a
Hilbert space.
We say that an operator $\cL = \cJ \cH$ densely defined on ${\rm X}$ is \emph{Hamiltonian} 
if the following hold:
\begin{enumerate}[label=\roman*)]
    \item the operator $\cJ : D(\cJ)\subset {\rm X}\to {\rm X}$ is densely defined, closed, and skew-adjoint;
    \item the operator $\cH : {\rm X}\to {\rm X}$ is bounded and self-adjoint.
\end{enumerate}
\end{definition}
For the Taylor-Green vortex studied in this paper, the linear operator $\cL_E$ in \eqref{cL} decomposes into
\begin{equation}\label{cJcH}
    \cL_E=\cJ\cH\,,\quad \text{with}\quad \cJ:=-\{\psi_E,\cdot\}\,,\quad \text{and}\quad \cH:=\uno+2\Delta^{-1}
\end{equation}
and is therefore Hamiltonian 
in the sense of Definition \ref{def:Hamiltonian}. Our general instability criterion then reads as follows.

\begin{theorem}
\label{thm:instability_criterion}
Consider a Hamiltonian operator $\cL=\cJ\cH$ as in Definition \ref{def:Hamiltonian}. Assume that $\cH$ admits a decomposition $\cH = \cH_s + \cH_u$ satisfying the following properties:
\begin{enumerate}[label=\arabic*)]
    \item $\cH_s:{\rm X}\to {\rm X}$ is self-adjoint and $\inf\{\langle \cH_s f,f\rangle_{\rm X} \, :\, \|f\|_{\rm X}=1\}>0$.
    \item $\cH_u : {\rm X}\to {\rm V}_u$ is
    of
    finite rank, with $n \coloneqq \mathrm{dim}({\rm V}_u) < +\infty$.
    \item $\cH_s$ and $\cH_u$ commute.
\end{enumerate}
Define the skew-adjoint operator $\cT := \cH_s^{1/2} \cJ \cH_s^{1/2}$.  For every $\lambda \in \CC\setminus \im\RR$, let $\cM_\lambda:{\rm V}_u \to {\rm V}_u$ be the finite-dimensional operator and $\Phi:\CC\setminus \im\RR \to \CC$ be the holomorphic function respectively  given by
\begin{equation}
\label{def:MPhi}
\cM_\lambda \coloneqq  \mathrm{Id} + \cH_u \cH_s^{-1} (\cT - \lambda)^{-1} \cT \, ,\qquad  \Phi(\lambda)\coloneqq \det(\cM_\lambda)\, .
\end{equation}
Then, 
$\mathrm{Ker}(\cL - \lambda)$ and $\mathrm{Ker}(\cM_\lambda)$ are isomorphic. In particular,
\begin{align}
    \lambda \in \sigma(\cL)\setminus \im\RR \quad \text{ if and only if } \quad  \Phi(\lambda)=0.
\end{align}
\end{theorem}

\begin{remark}

In the case $n = 1$, $\Phi$ takes a more explicit form. Indeed, let ${\rm V}_u={\rm span}\{E_u\}$ and  $h_u\in \RR$, $s_u > 0$ be such that $\cH_u(E_u)=-h_uE_u$, $\cH_s(E_u)=s_uE_u$. Then,  we have
\begin{equation} \label{eq:Phi_def}
     \Phi (\lambda) 
     =  1 - \frac{h_u}{s_u\|E_u\|_{\rm X}^2} \langle (\cT - \lambda)^{-1} \cT E_u, E_u \rangle_{\rm X},
\end{equation}
see \eqref{pf:n1Phi} for a full proof. 
\end{remark}

 {Let us comment on Theorem \ref{thm:instability_criterion}. This result is not perturbative in nature, in the sense that we are not relying on any smallness parameter to follow the onset of spectral instability.}
The relation between negative directions of the self-adjoint operator $\cH$ and unstable directions of the Hamiltonian operator $\cL = \cJ\cH$ has been extensively studied in the literature, both in the wave community \cite{GSS1,KapitulaPromislow2013,PegoWeinstein1992}  and in the fluid community \cite{FriedlanderVishik1991,FriedlanderVishik1992,friedlander1997nonlinear,Lin04IMRS,LinCMP2004}. While it is straightforward that non-definiteness of $\cH$ is a necessary condition for the onset of unstable eigenvalues of $\cL$ (see Lemma~\ref{lemma:uniqueness_algebra}), a general and sharp sufficient condition for instability  is much less transparent. In  \cite{LinZengMemoirs}, Lin and Zeng, building on the Pontryagin--Krein space theory \cite{Pontryagin1944,AzizovIokhvidov1989,GohbergLancasterRodman2005}, give an abstract comprehensive result for  Hamiltonian operators $\cL=\cJ\cH$, where the self-adjoint operator $\cH$ has finitely-many negative directions. They present a sevenfold decomposition of the domain such that  $\cL$ is represented as a block upper-triangular matrix of operators. Moreover, all the infinite-dimensional diagonal blocks are skew-symmetric, thereby reducing in principle the search for unstable eigenvalues to a finite-dimensional problem. 
They also provide a Galerkin-type construction that isolates a finite-dimensional invariant subspace in which the unstable dynamics (if any) is confined.
However, the application of their method would require to compute the spectrum of matrices of  larger and larger dimension (the Galerkin truncation of the operator) until a good approximate Pontryagin subspace is somehow reached. 
On the contrary, Theorem \ref{thm:instability_criterion} requires to work  on an a priori fixed and completely explicit space ${\rm V}_u$ whose finite dimension $n$ is exactly given by the number of negative directions of the operator $\cH$ (in the sequel we will always apply it with $n=1$ or $2$).

Under additional structural assumptions, Lin and Zeng's result \cite{LinZengMemoirs}  {recovers} instability criteria of \cite{lin2003instability,LinCMP2004}, and \cite[Lemma 3.1]{liao2023stability} and can be applied to a variety of contexts  \cite[Section 11]{LinZengMemoirs}, including instability results for solitary waves in modulational models and periodic traveling waves. However, in fluids these criteria are specifically designed to detect, under suitable assumptions, \emph{purely real eigenvalues}. This is a  decisive obstruction,  making them \emph{inapplicable} to the Taylor–Green vortex in \eqref{TaylorGreen11}. 

On the contrary, Theorem \ref{thm:instability_criterion} is a finite dimensional criterion and gives a simple and concrete characterization of the unstable eigenvalues as zeros of a holomorphic function. Therefore, it can be used both to rule out instability, and to actually find it, as we will do for the Taylor--Green vortex in different invariant spaces.  Our proof is self-contained and independent from the previously cited results.

On the other hand, it is interesting to observe that the proof of Theorem \ref{thm:instability_criterion} can be adapted to encompass, as a limiting case, the Latushkin-Vasudevan reformulation \cite{latushkin2018eigenvalues} of \cite[Theorem 1.2]{LinCMP2004}; see Remark \ref{rmk:backtoLin}.

In a nutshell, the splitting $\cH = \cH_s + \cH_u$ isolates a finite-dimensional space ${\rm V}_u = \mathrm{Rn} \,\cH_u$ containing all the possible negative directions of $\cH$. Then, for any eigenvalue $\lambda \in \CC \setminus \im \RR$, the kernel of the finite dimensional matrix $\cM_\lambda :{\rm V}_u\to {\rm V}_u $  is a lossless encoding of the negative parts of the eigenvectors associated with $\lambda$. The isomorphism from $\ker(\cM_\lambda)$ to the eigenspace $\Ker(\cL-\lambda)$, that we will define in \eqref{eq:isomorphism}, is exactly the operation of reconstructing the unstable directions, if any, from their negative projections. The resulting function $\Phi(\lambda) = \det(\cM_\lambda)$ is not canonical, since it depends on the choice of the splitting $\cH = \cH_s + \cH_u$. Theorem \ref{thm:instability_criterion} shows that, for any such admissible splitting, $\Phi$ yields a concrete analytic device whose zeros give a complete description of the spectrum of $\cL$ off the imaginary axis and of the associated multiplicities.

\subsection{Application of the instability criterion to the Taylor--Green vortex}
In order to apply this general criterion to our concrete operator $\cL_E$ in  \eqref{cL}, interpreted as \eqref{cJcH}, we first observe that the self adjoint operator $\uno+2\Delta^{-1}$ is a Fourier multiplier with four negative directions and four elements in the kernel. This implies (see Lemma \ref{lemma:uniqueness_algebra}) the existence of at most eight eigenvalues outside the imaginary axis.
To show that {only} four of them actually exist, we first {split the space in some} invariant subspaces of $\cL_E$ 
\begin{equation}\label{EvOddsplitting-intro}
L^2_0(\TT^2) =  \mathrm{Odd}_x\mathrm{Odd}_y  \oplus   \mathrm{Odd}_x\mathrm{Ev}_y \oplus
  \mathrm{Ev}_x\mathrm{Odd}_y \oplus  \mathrm{Ev}_x\mathrm{Ev}_y
\,,
\end{equation}
defined according to the parity with respect to the $x$ and $y$ variable separately. We can further split each of these invariant spaces in two, based on the fact that the Fourier decomposition of an element in each space, say $ \mathrm{Odd}_x\mathrm{Odd}_y $, is a sum of trigonometric monomials, multiples of $\sin(kx)\sin(jy)$, with $j+k$ even (resp. odd; see the precise definition in \eqref{harmonicsplitting} below). This gives rise to the decomposition in invariant subspaces  $\mathrm{Odd}_x \mathrm{Even}_y = (\mathrm{Odd}_x \mathrm{Even}_y )^{[\mathrm{ev}]} \oplus (\mathrm{Odd}_x \mathrm{Even}_y )^{[\mathrm{odd}]}\,,$ and analogously for the other subspaces in \eqref{EvOddsplitting-intro}.
We then observe that in some of these eight invariant subspaces, such as $ \mathrm{Odd}_x\mathrm{Odd}_y $ and $(\mathrm{Odd}_x \mathrm{Even}_y )^{[\mathrm{ev}]}$, $\cH$ has no negative directions; hence $\cL_E$ has no unstable eigenvectors therein.
In $(\mathrm{Odd}_x \mathrm{Even}_y )^{[\mathrm{odd}]}$ there can be only one {real unstable} eigenvalue. 
To rule out this possibility we show that $\Phi$ is monotonically increasing on the real positive semiaxis and we estimate its limit at $0$ from below with a positive constant (see Proposition~\ref{prop:OddEvstability}).
The latter estimate is enabled by Lemma \ref{lem:strong_convergence},  characterizing nice properties of the resolvent of a skew-adjoint operator. As a byproduct, said limit at $0$  involves 
the transport operator and the projection given by averaging over streamlines.
{Interestingly, if one considers a suitable rescaling of the Taylor--Green vortex, (see Section \ref{sec:nmTG} about the generalization of our result to $(m,n)$ Taylor--Green vortices), such instabilities corresponding to a real eigenvalue  and can be captured with Theorem~\ref{thm:instability_criterion}.}
 Finally, we consider the  $(\mathrm{Even}_x \mathrm{Even}_y )^{[\mathrm{odd}]}$ class, where $\cH$ has two negative directions. We introduce a further splitting into two $\cL_E$-invariant complex subspaces $ (\mathrm{Ev}_x \mathrm{Ev}_y)^{[\mathrm{odd}]}=(\mathrm{Ev}_x \mathrm{Ev}_y)^{[\mathrm{odd}]}_+ \oplus (\mathrm{Ev}_x \mathrm{Ev}_y)^{[\mathrm{odd}]}_- $ and we reduce the spectral problem to finding \emph{one} unstable eigenvalue in \emph{one} of these spaces. In particular, in the $(\mathrm{Ev}_x \mathrm{Ev}_y)^{[\mathrm{odd}]}_+$ class, we are able to localize the possible onset of an eigenvalue $\lambda_\star$, with $\Re \lambda_\star >0$, within the semicircle $|z+\frac\im2 | \leq \frac3{10} \sqrt{\frac56} \approx 0.27$. We then prove that such eigenvalue appears in the region $[0.10, 0.17] - \im [0.57, 0.63]$. This is verified with a computer assisted argument, showing that the holomorphic function $\Phi$ must have a zero in such region.

\medskip

\subsection{Related literature} \label{subsec:intro_literature}

The study of the spectral stability properties of  steady states {of the Euler equations} is a foundational topic within hydrodynamic stability theory \cite{drazin2004hydrodynamic}. Since the first experimental account of turbulence provided by Reynolds in 1883, many of the earlier investigations focused on understanding spectral \emph{instability} in parallel shear flows \cite{Rayleigh80,Tollmien1935,drazin2004hydrodynamic}. These instabilities have now been mathematically justified and understood in various settings, e.g. \cite{lin2003instability,Grenier16Duke,CDMV}. Moreover, unstable radially symmetric vortices (intimately related to shear flows) have been used to provide examples of forced nonuniqueness \cite{vishik2018instability,albritton2022non}. 

To study stability properties of more general steady states, in 1965 Arnold unveiled the Hamiltonian structure of the 2D Euler equations in his seminal paper \cite{Arnold65}, see also \cite{arnold2009topological}. He fully exploited this structure by deriving one of the first general nonlinear Lyapunov stability result for ideal fluids. In particular, the Hamiltonian function is the kinetic energy, and the motion is ``constrained" by the conservation of all Casimirs of vorticity. For steady states of the form $\omega_E=G(\psi_E)$, Arnold applied an elegant variational argument to derive a quadratic form\footnote{This quadratic form, always conserved by the linearized evolution, is defined through the operator $1/G'(\psi_E)+ \Delta^{-1}$. This is indeed proportional to our $\mathcal{H}$ for the Taylor-Green vortex.} whose fixed sign implies the stability of the steady state. The consequences of this result include the symmetry of the spectrum with respect to both the real and imaginary axes, the nonlinear stability of a class of shear flows\footnote{$U_E=(U(y),0)$ with $U''$ having a fixed sign in the domain.}, and a  stability criterion for steady states in general domains. Arnold's viewpoint has since been extended to many other conservative systems arising in the physics of fluids and plasmas \cite{holm1985nonlinear}, and has proven useful even in a relevant case involving viscosity 
 \cite{gallay2024arnold}.
{More broadly, the Hamiltonian framework has proved highly effective for the analysis of conservative dynamics, as witnessed by a wide range of results in both linear and nonlinear settings \cite{GSS1,GSS90,KapitulaPromislow2013,LinZengMemoirs,BBHM18}. From a spectral perspective, Hamiltonian operators naturally allow for genuinely non-normal spectral features that drive instability, such as Krein collisions, Jordan chains, and eigenvalues bifurcating off the imaginary axis. A prime example of the onset of these phenomena is famously provided by the spectrum of Stokes waves (travelling solutions of the system of water waves), whose modulational instability and related spectral scenarios have been progressively characterized; see \cite{BM95,NS23,BMV22,BMV23,BMV24,BCMV24}. }

We note that in recent years there has been growing interest in quantitative stability properties of stationary solutions, such as \emph{inviscid damping} of perturbations; see \cite{bedrossian2015inviscid,wei2019linear,bedrossian2019vortex,ionescu2023nonlinear,masmoudi2024nonlinear,lin2019metastability,grenier2020linear}. For the linearized problem, this mechanism corresponds to the weak convergence of the vorticity toward its average along the streamlines of the background flow (i.e.\ the projection onto $\Ker(\cJ)$), with precise polynomial decay rates for the velocity field. This property has so far been established only for shear flows or vortices, and is closely related to the \emph{absence} of embedded eigenvalues in the essential spectrum of the linearized Euler operator. Indeed, a weaker form of inviscid damping holds in fairly general settings under {further} assumptions \cite[Theorem~11.7]{LinZengMemoirs}. For concave shear flows and strictly monotone vortices, embedded eigenvalues can be excluded \cite{bedrossian2019vortex,wei2019linear}, whereas in more general situations this must be verified case by case \cite{lin2019metastability,grenier2020linear} or imposed as an assumption \cite{wei2019linear,ionescu2023nonlinear,masmoudi2024nonlinear}. In fact, characterizing embedded eigenvalues is a notoriously difficult problem; for Schr\"odinger-type operators, Mourre theory provides powerful tools to address this issue \cite{amrein1996commutator}, a connection exploited for certain shear flows in \cite{grenier2020linear}. By contrast, the theory for more general Hamiltonian systems (where $\cL$ is not skew-adjoint) is far less developed and, even in perturbative regimes, requires delicate analysis, see for instance \cite{hadzic2024absence}.

Finally, our proof uses computer-assistance in order to check that the function $\Phi(\lambda )$ vanishes at $\lambda = \lambda_\star$, concluding the instability part of Theorem \ref{linthm}. The code is written in Sage (\textit{SageMath 10-7}), and can be found in the supplementary material of the arXiv version of this paper ``TaylorGreen\_CAP.ipynb''. We refer to \cite{Tucker} as a general reference in interval arithmetics, to \cite{GomezSerrano} for a recent survey focused on PDE, and in particular to \cite[Section 4]{GomezSerrano} for the application of interval arithmetics to show the existence of unstable eigenvalues for linear operators arising in the analysis of fluids.

Interval arithmetics, going back to the pioneering work of R. E. Moore \cite{Moore} consists on the idea of replacing real numbers by intervals in order to prove rigorous statements with the computer. For example, in order to compute the sum of two real intervals $a^I = [a^L, a^R]$, $b^I = [b^L, b^R]$, one can implement the sum using that
$a^I + b^I \subset [(a^L + b^L)_-, (a^R + b^R)_+]$, where $a_\pm$ are the operators rounding down/up to the nearest representable number from below/above. Although this inclusion is not an equality, the containment property suffices to ensure rigor: to show $a+b \neq 0$ for all $a\in a^I$, $b\in b^I$, it suffices to show that $0 \notin [(a^L + b^L)_-, (a^R + b^R)_+]$. One can perform much more complicated operations with interval arithmetics, such as integrals, which is classical since the work of Moore \cite{Moore}.

Early applications of interval arithmetics to PDE stability problems were related to the verification  of certain numerical inequalities, as in the relativistic stability of matter by De La Llave -- Fefferman \cite{DeLaLlaveFefferman}. However, the recent increase in computational power has allowed to treat much more computationally expensive problems. Stability has been studied with computer-assistance via Evans functions \cite{BarkerZumbrun}, by combining approximate inversion with fixed point theorems \cite{Cadiot, WPN}, or reducing stability to properties of explicit solutions \cite{BCG}. In the last decade, an approach has gained significant attention: loosely speaking, we decompose a linear operator $\cL$ into a finite-rank part, analyzed with the computer, and a remainder with controlled influence on the spectrum. This has been employed to construct solutions for the SQG equation via Crandall-Rabinowitz \cite{CastroCordobaGomezS}, or to prove singularities for the 3D Euler equations with boundary \cite{ChenHou1, ChenHou2}. We understand our method as a new way of performing this splitting, which is particularly suited for Hamiltonian systems. A remarkable difference with previous iterations of this idea is that our method reduces the problem to study a single Evans function $\Phi(\lambda)$ defined in \eqref{def:MPhi}, significantly simplifying the problem.

We also recall that our unstable mode at $\lambda = \lambda_\star$ is expected to be of low regularity (see Remark~\ref{rmk:nonsmooth}). Finding low regularity solutions to PDE has traditionally been out of reach for computer-assisted methods due to the difficulty of obtaining accurate representations. Recent methods to address this problem rely on using a Clausen functions basis, which is adapted to the low regularity of the solution \cite{EncisoGomezSVergara, DahneGomezSerrano, Dahne}. In contrast, our stability criterion allows us to work directly in a Fourier basis at low regularity, since accurate representations of the solution are no longer required.

\section{Stability in linear Hamiltonian operators} \label{sec:hamiltonian}

In this section, we analyze the stability properties of a linear Hamiltonian operator $\cL= \cJ \cH$, with $\cJ$ and $\cH$ as in Definition \ref{def:Hamiltonian}.
We first present some preliminary results relating the dimension of the  \textit{negative directions} of $\cH$ to the \textit{unstable directions} of $\cL$. These type of results are well established in the literature in rather general cases, e.g. \cite{KapitulaPromislow2013, LinZengMemoirs}, but here we provide short and self-contained proofs that suffice for our purposes. These properties will be used later in Section \ref{sec:stability} to characterize the stability of the Taylor-Green vortex under perturbations within a certain symmetry class.

Then, we present the proof of Theorem \ref{thm:instability_criterion}. Finally, we study useful properties of the \textit{reduced stability matrix} $\cM_\lambda$ and the \textit{stability function} $\Phi(\lambda)$ defined in \eqref{def:MPhi}. These properties are our building blocks in the proof of the stability with respect to even-odd or odd-even perturbations of the Taylor-Green vortex.

\subsection{Negative directions and instability}
For a linear Hamiltonian operator, it is straightforward to see that the quadratic form $q_\cH:f\mapsto\langle f,\cH f\rangle_{\rm X}$ is a conserved quantity of the dynamics. In particular,
if $\cH$ is coercive then $q_\cH$ is a Lyapunov functional for the linear dynamic, and $\cL$ has no unstable eigenvalues. Only the presence of both {\it positive} and \textit{negative }directions of $\cH$ can potentially generate unstable modes for the whole $\cL=\cJ\cH$. 
Moreover, in rather general cases (cfr.\ \cite[Theorem 2.2]{LinZengMemoirs}), the number of unstable directions of $\cL$ is bounded by the number of negative directions of $\cH$.   
In our functional setting this is not hard to prove as shown below.
\begin{lemma} 
\label{lemma:uniqueness_algebra}
    Let $\cL = \cJ \cH$ be a Hamiltonian  such that
    \begin{enumerate}[label=\arabic*)]
        \item $\cH$ 
        has exactly $k$ negative eigenvalues (counted with multiplicity);
        \item  $\cL$ has $m$ linearly independent unstable eigenfunctions.
    \end{enumerate}
    Then  any unstable eigenvector $f$ of $\cL$ must satisfy $\langle \cH f, f \rangle_{\rm X} = 0$ and one has 
    $ m \leq k$.
    In particular, if $\cH$ has no negative eigenvalues, then $\sigma(\cL) \subseteq \im \mathbb R$. 
\end{lemma}

\begin{proof}
    Let $\{f_j\}_{j=1}^m$ of $\cL$ be the eigenfunctions of $\cL$ with eigenvalues $\{\lambda_j\}_{j=1}^m$ such that $\Re(\lambda_j) > 0$ for all $j$, and let $S = \mathrm{span}\{f_1, \dots, f_m\}$. We  first show that $\langle\cH \cdot, \cdot\rangle$ vanishes on $S$. We notice that for any $p, q = 1, 2, \dots , m$ (possibly equal), we have:
    \begin{equation} \label{eq:H_orthogonal}
        \lambda_p \langle f_p, \cH f_q \rangle = \langle \cJ \cH f_p, \cH f_q \rangle = - \langle \cH f_p, \cJ \cH f_q \rangle = - \overline{\lambda}_q \langle \cH f_p, f_q \rangle.
    \end{equation}
   Hence $(\lambda_p + \overline{\lambda}_q) \langle \cH f_p, f_q \rangle = 0$. On the other hand, since $\Re(\lambda_p), \Re(\lambda_q) > 0$, the sum $\lambda_p +\overline{\lambda}_q $ does not vanish. Thus, $\langle \cH f_p, f_q \rangle = 0$ for all $p, q =1,\dots,m$ and the bilinear form associated with $\cH$ vanishes identically on the $m$-dimensional subspace $S$. To conclude, let $\Pi_+$ and $\Pi_-$ be the orthogonal projection respectively onto the maximal positive and negative space of $\cH$. We claim that the vectors $\Pi_- f_1,\dots,\Pi_- f_m$ are linearly independent. If not, then there would exist $S\ni f = \alpha_1 f_1 + \dots + \alpha_m f_m \neq 0$ such that $\Pi_- f =\alpha_1 \Pi_- f_1 + \dots + \alpha_m \Pi_- f_m $ vanishes. Since $\langle \cH f, f \rangle = 0$, this implies that $\Pi_+ f = 0$, so we conclude $f \in \mathrm{Ker}(\cH) \subset \mathrm{Ker}(\cL)$. That is a contradiction, since $f \neq 0$ and $f$ is a linear combination of $f_i$, which are eigenfunctions whose eigenvalues have positive real part.
    This proves our claim that $\{ \Pi_- f_i \}_{i=1}^m$ are independent, whence $m = \mathrm{dim}(S) \leq \mathrm{dim}\big(\Pi_-({\rm X})\big) = k$.
\end{proof}

\begin{remark}[Negative and kernel directions of $\cH$]
    In the case of the Taylor-Green vortex, one can readily see that the self-adjoint operator $\cH=\uno+2\Delta^{-1}$ in \eqref{cJcH} has a four-dimensional negative subspace given by
\begin{equation}\label{negativedirectionscH}
    \mathrm{N}:=\mathrm{span}\big\{\sin(x)\,,\sin(y)\,,\cos(x)\,,\cos(y) \big\} \, ,
\end{equation}
and a four-dimensional kernel given by
\begin{equation}\label{kernelcH}
\mathrm{K}:=\mathrm{span}\big\{\sin(x)\sin(y)\,,\sin(x)\cos(y)\,,\cos(x)\sin(y)\,,\cos(x)\cos(y) \big\}\, .
\end{equation}
On the orthogonal space $(\mathrm{N} \oplus \mathrm{K})^\perp $, the operator $\cH$ is positive-definite and coercive. Indeed, if $f\in (\mathrm{N} \oplus \mathrm{K})^\perp$ then the Fourier coefficients $\hat{f}(j,k)=0$ for $(j,k)\in \{(0,0),(0,\pm 1), (\pm 1,0), (\pm 1,\pm 1)\}$, meaning that on the Fourier support of $f$ one has $j^2+k^2\geq 5$. Hence, being $\cH$ a Fourier multiplier explicitly given by $1-2/(j^2+k^2)$, we get the coercivity bound for $\cH$ on its stable subspace
\begin{equation}\label{coercivewhole}
    \langle \cH f, f\rangle_{L^2}=\sum_{k,j\in \ZZ}\mathbbm{1}_{\{j^2+k^2\geq 5\}} \Big(1-\frac{2}{j^2+k^2}\Big)|\hat f(j,f)|^2\,  \geq \,\frac35 \|f\|_{L^2}^2 \,,\quad \forall f \in (\mathrm{N} \oplus \mathrm{K})^\perp\, .
\end{equation}
Thus, by Lemma \ref{lemma:uniqueness_algebra}, the operator $\cL$ in \eqref{cL} has at most four unstable eigenvalues. Most of the paper is dedicated to show that such unstable eigenvalues are actually exactly two, as stated in Theorem~\ref{linthm}.
\end{remark}

\subsection{Proof of the instability criterion}
We now turn our attention to the proof Theorem \ref{thm:instability_criterion}, which is based on an explicit construction of the desired isomorphism between $\Ker(\cL-\lambda)$ and $\Ker(\cM_\lambda)$.
\begin{proof}[Proof of Theorem \ref{thm:instability_criterion}]
    Let $f\in \Ker(\cL-\lambda)$. By decomposing $\cL=\cJ(\cH_s+\cH_u)$, we get
\begin{align}\label{eq:JHsHu}
    (\cJ\cH_s-\lambda) f=-\cJ\cH_uf.
\end{align}
Since $\cH_s$ is coercive and self-adjoint,  by standard functional calculus, there exists a unique coercive square root operator $\cH_s^\frac12$, and it commutes with $\cH_u$.
By applying the operator $\cH_s^\frac12$ to the identity \eqref{eq:JHsHu}, 
\begin{align}
\label{eq:fT}
    (\cT-\lambda)(\cH_s^\frac12 f)=-\cT\cH_u(\cH_s^{-\frac12} f).
\end{align}
Since $\Re\,\lambda\neq 0$ and $\cT$ is skew-adjoint, we can invert $(\cT-\lambda)$. We apply first $(\cT-\lambda)^{-1}$  and then the operator $\cH_u\cH_{s}^{-1}$ to both sides of \eqref{eq:fT}, to deduce that 
\begin{align}\label{compression}
    \cH_u \cH_s^{-\frac12} f+\cH_u\cH_s^{-1}(\cT-\lambda)^{-1}\cT\big(\cH_u\cH_s^{-\frac12}f\big)=0.
\end{align}
The identity above readily implies that if $f\in \Ker(\cL-\lambda)$ then $z\coloneqq \cH_u\cH_s^{-\frac12}f\in \Ker(\cM_{\lambda}).$ 

Now, let $z\in \Ker(\cM_\lambda)$, meaning that 
\begin{align}
    z+\cH_u\cH_s^{-1}(\cT-\lambda)^{-1}\cT z=0.
\end{align}
Then, let $f\coloneqq \cH_s^{-\frac12}(\cT-\lambda)^{-1}\cT z$. Since $\cH_s^\frac12$ is positive definite, by \eqref{eq:JHsHu} we get that $f\in \Ker(\cL-\lambda)$ if and only if \eqref{eq:fT} is satisfied. With the choice of $f$ as before, we compute that \begin{align}
     (\cT-\lambda)(\cH_s^\frac12 f)=\cT z, \qquad \cT\cH_u(\cH_s^{-\frac12}f)=
    \cT\cH_u\cH_s^{-1}(\cT-\lambda)^{-1}\cT z.
\end{align}
Hence, 
\begin{align}
  (\cT-\lambda)(\cH_s^\frac12 f)+\cT\cH_u(\cH_s^{-\frac12}f)=\cT\big(z+\cH_u\cH_s^{-1}(\cT-\lambda)^{-1}\cT z)=0,
\end{align}
where in the last identity we used that $z\in \Ker(\cM_\lambda)$, and therefore $f\in \Ker(\cL-\lambda)$ as desired.

All in all, we can write explicitly the linear isomorphisms of vector spaces, inverse of each other, as
\begin{alignat}{2} \label{eq:isomorphism}
&\Ker(\cL-\lambda) \to \ \mathrm{Ker}(\cM_\lambda) \qquad &&\mathrm{Ker}(\cM_\lambda)  \to \Ker(\cL-\lambda)\\
&f \mapsto \ \cH_u \cH_s^{-1/2} f   &&z   \mapsto  -\cH_s^{-1/2}(\cT - \lambda)^{-1} \cT z.
\end{alignat}

Finally, we study the $n = 1$ case. Since $\cH_u$ is self-adjoint and rank-$1$, we know that there exists a vector $E_u$ and $h_u\in \RR$  such that  $$\cH_u = -\frac{h_u}{\|E_u\|_{\rm X}^2} (E_u \otimes E_u), \quad \Longrightarrow\quad\cH_u(E_u)=-h_uE_u$$ 
Since $\cH_u$ and $\cH_s$ commute, we know that $\cH_u(\cH_s(E_u))=\cH_s(\cH_uE_u)=-h_u\cH_s(E_u)$. Therefore, being $\cH_s$ a positive definite self-adjoint operator we infer that there exists $s_u>0$ such that $\cH_s(E_u)=s_u E_u$. Moreover, note that $s_u - h_u = \langle (\cH_s + \cH_u)E_u, E_u \rangle = \langle \cH E_u, E_u \rangle$. Finally
\begin{align}
\label{pf:n1Phi}
\Phi (\lambda) &= \mathrm{det}(\cM_\lambda) = \frac{\langle \cM_\lambda E_u, E_u \rangle_{\rm X}}{\|E_u\|^2_{\rm X}} = 1 +\frac{1}{\|E_u\|_{\rm X}^2} \langle (\cT - \lambda)^{-1} \cT E_u, \cH_s^{-1} \cH_u E_u \rangle_{\rm X} \\
&= 1 - \frac{h_u}{s_u} \frac{1}{\|E_u\|_{\rm X}^2}\langle (\cT - \lambda)^{-1} \cT E_u, E_u \rangle_{\rm X},
\end{align}
which proves the desired identity \eqref{eq:Phi_def}.
\end{proof}

\subsection{Properties of the stability determinant on \texorpdfstring{$\mathbb{R}$}{R}}

Let us consider now a \textit{real} Hamiltonian operator $\cL$ fulfilling Definition \ref{def:Hamiltonian} on a real Hilbert space $\rm{X}$.
In view of Theorem \ref{thm:instability_criterion}, we reduced the problem of finding unstable modes for a Hamiltonian operator $\cL$ to finding zeros of the \textit{stability determinant} $\Phi(\lambda)$. It is first natural to ask under which circumstances $\Phi$ has \emph{real} zeros. Notice in fact that $\Phi(\lambda) \in \mathbb{R}$ for every $\lambda \in \mathbb R$ since $\cT$ is real. Moreover, since $\cT$ is skew-adjoint, we have $\| (\cT - \lambda)^{-1} \|_{{\rm X}} \leq \frac{1}{| \mathrm{Re} (\lambda )|}$ and
\begin{equation}
    \label{eqn:phi1}
    \lim_{|\lambda |\to +\infty} \Phi (\lambda) = 1.
\end{equation}
We then need to understand whether $\Phi$ has limit as $\lambda \to 0^+$; since, if the latter happens to be negative, we can automatically infer the existence of at least one real unstable eigenvalue, by relying on the continuity of $\Phi$ on the positive real semiaxis. 
Remarkably, the limit  as $\lambda \to 0^+$ of the matrix $\cM_{\lambda}$ in \eqref{def:MPhi} is always well-defined.
\begin{proposition}[Limit at zero of $\cM_\lambda$]\label{prop:limit_zero}
Under the same hypotheses of Theorem \ref{thm:instability_criterion},  the finite-rank operator $\cM_\lambda$ admits the limit
\begin{equation}\label{eq:limit_M_0}
    \lim_{\lambda \to 0^+} \cM_\lambda = \uno + \cH_u \cH_s^{-1} (\uno-P_{\cT}) =: \cM_0\,,
\end{equation}
where $P_{\cT}$ is the orthogonal projection onto the kernel of $\cT:= \cH^{1/2} \cJ \cH^{1/2}$ and  can be expressed as 
\begin{equation}\label{fromPTtoP}
    P_{\cT}=\cH_s^{-\frac12}P(P\cH_s^{-1}P)|_{\Ker(\cJ)}^{-1}P\cH_s^{-\frac12}\, ,
\end{equation}
where $P$ is the orthogonal projection onto the kernel of $\cJ$. In particular, if $n=1$ then
\begin{equation}\label{Phin=1}
\lim_{\lambda \to 0^+}\Phi(\lambda):=\Phi(0)=1+\frac{h_u}{s_u}\Big(\frac{1}{s_u\|E_u\|^2_{\rm X}}\langle (P\cH_s^{-1}P)^{-1}PE_u,PE_u\rangle_{\rm X}-1\Big)\, .
    \end{equation}
\end{proposition}
Observe that the right-hand side in \eqref{Phin=1} always returns a real number, because $\cH_u$, $\cH_s$ and $(P\cH_s^{-1}P)^{-1}$ are self-adjoint operators.

The proof of Proposition \ref{prop:limit_zero} boils down to computing the limit as $\lambda \to 0^+$ of $(\cT-\lambda)^{-1}\cT$, for a general skew-adjoint operator $\cT$. This is achieved by Lemma \ref{lem:strong_convergence} below. In the case of the transport operator $\cJ=\{\psi,\cdot\}=\nabla^\perp\psi\cdot \nabla$, with  $\psi\in C^2(\TT^2)$, the limit was carried out by Z.\ Lin in \cite{Lin04IMRS} and recently revisited in \cite{latushkin2018eigenvalues}. In particular, it was shown that $(\cJ-\lambda)^{-1}\cJ\to \uno -P$ as $\lambda \to 0^+$, where $P$ is the orthogonal projection onto the kernel of $\cJ$ and the convergence is strong in $L^2_0(\TT^2)$. In general, we are able to show
the following result, that roots back to \cite[ Section III.6.5]{Kato1995}. 
\begin{lemma}[Strong Convergence of Resolvent Projection] \label{lem:strong_convergence}
    Let $\cT$ be a closed, densely defined, skew-adjoint operator on a Hilbert space ${\rm X}$. Let $P_{\cT}$ denote the orthogonal projection onto the kernel of $\cT$. Then, the operator family converges strongly as $\lambda \to 0^+$:
    \begin{equation}
        \lim_{\lambda \to 0^+} (\cT - \lambda)^{-1} \cT = \uno - P_{\cT} \,.
    \end{equation}
\end{lemma}
\begin{proof}
    Since $\cT$ is skew-adjoint, by the spectral theorem for skew-adjoint operators \cite[Theorem VIII.4]{ReedSimonI}, there exists a unique projection-valued measure $\mu$ on $\mathbb{R}$ such that $\cT = \int_{\mathbb{R}} \im\nu \, d\mu(\nu)$. The identity operator is given by $\uno = \int_{\mathbb{R}} d\mu(\nu)$, and the projection onto the kernel corresponds to $P_{\cT} = \mu(\{0\})$.
    
    Using functional calculus, the operator $(\cT - \lambda)^{-1} \cT$ is defined by the integral with respect to the measure $\mu$ of the function
    \begin{equation}
        h_\lambda(\nu) = \frac{\im\nu}{\im\nu - \lambda} \,.
    \end{equation}
   The sequence $h_\lambda(\nu)$ converges pointwise to the indicator function $\chi_{\mathbb{R} \setminus \{0\}}(\nu)$ as $\lambda \to 0^+$ for $\nu \in \mathbb{R}$.
     Furthermore, since $\lambda,\nu\in\RR$, we observe that the following uniform bound holds
    \begin{equation}
        |h_\lambda(\nu)|^2 = \frac{\nu^2}{\nu^2 + \lambda^2} \leq 1 \quad \text{for all } \nu,\, \lambda \in \RR\, .
    \end{equation}
    For any vector $f \in {\rm X}$, let $\mu_f$ be the scalar spectral measure defined by $\mu_f(B) := \langle \mu(B)f,f\rangle_{\rm X}$ for any Borel set $B$ of $\RR$. Since the constant function $1$ is integrable with respect to the finite measure $\mu_f$, we apply the dominated convergence theorem
    \begin{equation}
        \lim_{\lambda \to 0^+} \| (\cT-\lambda)^{-1}\cT f - (\uno - P_{\cT}) f \|_{{\rm X}}^2 = \lim_{\lambda \to 0^+} \int_{\mathbb{R}} \left| h_\lambda(\nu) - \chi_{\mathbb{R} \setminus \{0\}}(\nu) \right|^2 \, d\mu_f(\nu) = 0 \,.
    \end{equation}
    Then $(\cT-\lambda)^{-1}\cT$ converges strongly to the operator 
    $I - P_{\cT}$.
\end{proof}
With Lemma \ref{lem:strong_convergence}, the proof of Proposition \ref{prop:limit_zero} is reduced to proving the identity \eqref{fromPTtoP}.
\begin{proof}[Proof of Proposition \ref{prop:limit_zero}]
By combining Lemma \ref{lem:strong_convergence} with \eqref{def:MPhi}, the proof of \eqref{eq:limit_M_0} is straightforward. To prove the identity \eqref{fromPTtoP},
 let $\Pi$ be its right-hand side. This is a well-defined operator, since $P\cH_s^{-1}P $ is invertible on $\Ker(\cJ)= \Rn(P)$.
 {Using that $\cH_s$ is bijective, we have that} $P \cH_s^{-\frac12}$ maps the whole space $\rm X$ into $\Ker(\cJ)$ surjectively. Moreover, one readily checks that $\Pi^2 = \Pi = \Pi^*$,
namely $\Pi$ is an orthogonal projection. It now suffices to 
show that $\Rn(\Pi) = \Ker(\cT)$.
This follows from the fact that $\Ker(\cT) = \cH_s^{-\frac12} \Ker(\cJ)$, which is a direct consequence of the definition of $\cT$.
\end{proof}

Note that, in the case where $\cJ=\{\psi,\cdot\}$ is the transport operator and $\psi$ is sufficiently regular, $P$ can be easily characterized via the  average along the streamlines (connected components of level sets) of $\psi$ \cite[Lemma 2.3]{LinCMP2004}. Namely, let $\{\psi=h\}=\cup_{j=1}^{n_h}\Gamma_j(h)$ where $h$ is a regular value for $\psi$. Then, for any $(x,y)\in \Gamma_j(h)$ we have
\begin{equation}
    \label{eqn:projectionkerJ}
(Pf)(x,y)=\frac{1}{\oint_{\Gamma_j(h)}\frac{d\sigma}{|\nabla \psi|}}\oint_{\Gamma_j(h)}f\frac{d\sigma}{|\nabla \psi|}.
\end{equation} By Sard's theorem, whenever $\psi$ is smooth enough we can ignore critical points since the level sets associated to them have measure zero, meaning that the formula above fully characterizes the projection $P$ in $L^2$. We will exploit this characterization of the average in Section \ref{subsec:difficult_stability}.

As a consequence of Proposition \ref{prop:limit_zero}  and of Theorem \ref{thm:instability_criterion} we have the following.
\begin{corollary}
[Instability criterion for real eigenvalues]\label{lem:weakinstability} Let $\cL = \cJ \cH$ be a Hamiltonian operator satisfying the assumptions of Theorem \ref{thm:instability_criterion}. If the maximal negative invariant subspace ${\rm V}_-$ of $\cH$ has finite dimension $k$, with $k$ odd, and $P$ vanishes identically on ${\rm V}_-$ then $\cL$ has at least one real unstable  eigenvalue.
\end{corollary}
\begin{proof} Under the hypotheses of Theorem \ref{thm:instability_criterion}, it is always possible to choose a splitting $\cH=\cH_s+\cH_u$ such that
${\rm V}_u := {\rm V}_- \oplus \Ker(\cH) $, with $-\cH_u$ coercive on ${\rm V}_u$. We claim that the value $\Phi(0):= \det \cM_0$, with $\cM_0$ defined in \eqref{eq:limit_M_0}, is negative. Indeed, by exploiting that $\cH_s$ and $\cH_u$ are both invertible on ${\rm V}_u$ and commute, we conjugate  $\cM_0$ through $\cH_s^{\frac12} (-\cH_u)^{-\frac12} $ to obtain the self-adjoint operator 
$$
\cS: {\rm V}_u \to {\rm V}_u\,,\quad \cS := \uno -  \cH_s^{-\frac12} (-\cH_u)^{\frac12} (\uno - P_{\cT}) (-\cH_u)^{\frac12} \cH_s^{-\frac12} \, .
$$
 Since $\cH_s$ is coercive, the signature of $\cS_0$ is the same as the self-adjoint operator $\Sigma: {\rm V}_u \to {\rm V}_u$ given by
$$
\Sigma:= \cH_s^{\frac12} \cS \cH_s^{\frac12} = \cH_s -  (-\cH_u)^{\frac12} (\uno - P_{\cT}) (-\cH_u)^{\frac12}  = \cH +  (-\cH_u)^{\frac12} P_{\cT} (-\cH_u)^{\frac12}  =: \cH + \cP\, ,
$$
where  $\cP$ is positive-semidefinite.
For every $ u = v + w\,,\; u' = v'+w' \in {\rm V}_u $, with $v, v' \in {\rm V}_-$ and $w, w' \in \ker \cH$, one has 
$$
\langle \Sigma u, u' \rangle = \langle \cH v , v' \rangle + \langle \cP w, w' \rangle\,,
$$
because $\cH w = \cH w' = 0$ and $\cP v = \cP v' = 0$, where the latter identities follow from \eqref{fromPTtoP} and the hypothesis $P|_{{\rm V}_u}\equiv 0$.
In particular the maximal negative eigenspace of $\Sigma $ is ${\rm V}_-$, which has odd dimension. Consequently, $\cS$ has an odd number of negative directions and $\Phi(0) = \det \cM_0 = \det \cS < 0$ as claimed. By \eqref{eqn:phi1} we conclude that $\Phi$ has to vanish at least once on the positive real semiaxis.
\end{proof}
{While $\Phi(0)<0$ is a sufficient condition for real instabilities, it is not necessary in general. It becomes necessary when $\cH$ has only one negative direction, thanks to} 
the 
following monotonicity property.
\begin{proposition}[Monotonicity of $\Phi$]\label{prop:propertiesofPhi}
    Assume the same hypotheses of Theorem \ref{thm:instability_criterion} with $n=1$ and $h_u>0$.
    Then, the restriction
    $\Phi|_{\mathbb{R} \setminus \{ 0 \} } $ is a real even function, increasing
    in $| \lambda |$. 
\end{proposition}

\begin{proof}
From the definition of $\Phi(\lambda)$ in \eqref{eq:Phi_def}, since $h_u/s_u>0$, it is enough to show that
    \begin{equation}
        I(\lambda) := \langle (\cT-\lambda)^{-1} \cT E_u, E_u \rangle_{{\rm X}}
    \end{equation}
    is decreasing.
    Let us consider, as we did in the proof of Lemma \ref{lem:strong_convergence},
    the
    projection-valued measure $\mu$ on $\RR$ such that $\cT = \int_{\RR} \im \nu \, d\mu (\nu)$, and  $\mu_{E_u}(B) := \langle \mu(B)E_u, E_u \rangle_{\rm X}$, for every Borel set $B$ of $\RR$. Then
    \begin{equation} \label{eq:spectral_int}
        I(\lambda) = \int_{\RR} \frac{\im\nu}{\im\nu - \lambda} \, d\mu_{E_u}(\nu) = \int_{\RR} \frac{\im\nu(-\im \nu - \lambda)}{\nu^2 + \lambda^2} \, d\mu_{E_u}(\nu) = \int_{\RR} \frac{\nu^2 - \im \nu\lambda}{\nu^2 + \lambda^2} \, d\mu_{E_u}(\nu).
    \end{equation}
    Moreover, the operator $\cT$ is real, and $I(\lambda)$ is real as well. Alternatively, one can argue that since $\nu \in \mathbb R$ and $\cT$ is real, the spectral measure $\mu_{E_u}$ is even, i.e.\ $\mu_{E_u}(B) = \mu_{E_u}(-B)$ for every Borel set of $\RR$. The imaginary part in the integrand of \eqref{eq:spectral_int} is odd in $\nu$ and its integral vanishes. As a consequence
    \begin{equation}
        I(\lambda) = \int_{\RR} \frac{\nu^2}{\nu^2 + \lambda^2} \, d\mu_{E_u}(\nu).
    \end{equation}
    We conclude that $I(\lambda)$ is strictly decreasing, and $\Phi(\lambda)$ is strictly increasing, with respect to $|\lambda |$.
\end{proof}

\begin{remark}[A lower bound for $\Phi(0)$] \label{rem:n1phi0bis}
    To rule out positive real eigenvalues through the previous proposition, one needs only to justify that $\Phi(0)>0$. To this end, we lower bound identity \eqref{Phin=1}
   and choose the constant $s_u$ conveniently (and hence $h_u$). We observe that  since $P\cH_s^{-1}P$ is self-adjoint and bounded on ${\rm X}$, then for $f\in \Ker(\cJ)$ with $\|f\|_{\rm X}=1$ it holds that
\begin{align} \begin{split}
\label{def:coercivityRem}
    \langle (P\cH_s^{-1}P)^{-1}f,f\rangle_{\rm X} &= \|(P\cH_s^{-1}P)^{-1/2}f \|^2_X \\
    &\geq \frac{1}{\|P\cH_s^{-1}P\|_{{\rm X}\to{\rm X}}}\geq \inf_{\|g\|_{\rm X}=1}\langle \cH_sg,g\rangle_{{\rm X}} =: {\rm c}(\cH_s)>0.
\end{split} \end{align}
Indeed the second inequality is justified since $\cH_s$ is positive and $\|P\cH_s^{-1}P\|_{{\rm X}\to{\rm X}}\leq \|\cH_s^{-1}\|_{{\rm X}\to{\rm X}}\leq 1/{\rm c}(\cH_s)$, where this follows from ${\rm c}(\cH_s)\|g\|^2\leq \langle \cH_sg,g\rangle\leq \|\cH_sg\|\|g\| $ by replacing $g=\cH_s^{-1}h$. Then, since ${\rm V}_u={\rm span}\{E_u\}$ and $\cH_s|_{{\rm V}_u}=s_u\uno_{{\rm V}_u}$, we know that  ${\rm c}(\cH_s)=\min \{{\rm c}(\cH|_{{\rm V}_u^\perp}), s_u\}$, where ${\rm c}(\cH|_{{\rm V}_u^\perp})$ is defined as in \eqref{def:coercivityRem}. Thus, choosing $s_u={\rm c}(\cH|_{{\rm V}_u^\perp})$ we get that
\begin{equation}
\label{eq:achoice}
    \Phi(0)\geq 1+\frac{h_u}{s_u}\Big(\frac{\|PE_u\|_{\rm X}^2}{\|E_u\|_{\rm X}^2}-1\Big).
\end{equation}
The choice of this particular parameter $s_u$ is exactly the one we will implement in 
Section \ref{sec:stability}. 
\end{remark}

We conclude this section with the following.
\begin{remark}[Comparison with Lin's criterion]\label{rmk:backtoLin} We observe that, in all the steps of the proof of Theorem \ref{thm:instability_criterion}, the operator $\cH_u$ (and correspondingly the full $\cH$) does not need to be self-adjoint but one can relax the assumptions to operators of the form $\uno+\cK$ with $\cK$ compact. This is the key to construct a bridge between our theory and the instability criterion established by Lin in \cite[Theorem 1.2]{LinCMP2004}, later revisited in \cite{latushkin2018eigenvalues}.  
They consider a general 
2D-Euler steady state in the form 
$\omega_0 = -g(\psi_0) $, with $\psi_0 := \Delta^{-1} \omega_0$ and a smooth function $g$. The linearized Euler operator at $\omega_0$ is given by
$$\cL_{0} :=  -\{\psi_0, \uno  + g'(\psi_0) \Delta^{-1}\} =  \cJ_0 \cH_0\, ,$$ 
with $\cJ_0 := -\{\psi_0,\cdot\}
$  and 
$ \cH_0 := \uno + g'(\psi_0) \Delta^{-1}
$. Unless $g$ is strictly monotone,  the operator $\cL_0$ is no longer Hamiltonian in the sense of Definition \ref{def:Hamiltonian}.
Nonetheless, by splitting 
$\cH$ into 
\begin{equation}
    \label{splittingLin}
\cH_s := \uno\,, \quad \cH_u := g'(\psi_0) \Delta^{-1}\, ,
\end{equation}
it is not hard to adapt the proof of Theorem \ref{thm:instability_criterion} and obtain
that $\lambda \in \CC\setminus \im \RR$ is an $L^2(\TT^2)$-eigenvalue of $\cL_0$  if and only if the operator
$$  \widetilde\cM_\lambda:= \uno + g'(\psi_0) (\cJ_0-\lambda)^{-1} \cJ_0 \Delta^{-1} 
$$ 
has nontrivial kernel. 
Note that to obtain the statement above 
we used that $\cJ_0$ and $g'(\psi_0)$ commute.
{Let us observe, moreover, that  $ \widetilde\cM_\lambda$ is not a finite-dimensional operator, in contrast to Theorem \ref{thm:instability_criterion}, but a compact perturbation of the identity. Indeed,}
the operator $\uno - \widetilde\cM_\lambda$ is exactly the (Birman-Schwinger type) compact operator  $K_\lambda(0)$  used in \cite{latushkin2018eigenvalues}. Here the unstable spectrum of $\cL_0$ is related to the zero set of the two-modified Fredholm determinant of $\widetilde\cM_\lambda$. A direct consequence of such interpretation is that if $\widetilde\cM_0$, defined according to  Lemma \ref{lem:strong_convergence} as
$$
\widetilde\cM_0:=\lim_{\lambda \to 0^+} \widetilde \cM_\lambda = \uno +  g'(\psi_0)\big(\uno - P_{\rm Ker(\cJ_0)}\big)\Delta^{-1}  \, ,
$$
 has an odd number of negative directions, then there exist an unstable real eigenvalue of $\cL_0$. This is the analogue of Corollary \ref{lem:weakinstability}. 
\end{remark}

\subsection{Orbital stabilty
}

{
We now give a general criterion to establish the orbital stability of a Hamiltonian system $\de_t f=\cL f$ with $\cL=\cJ\cH$ as in Definition \ref{def:Hamiltonian}. To this end, one has to exclude the presence of unstable eigenvalues for $\cL$ as well as Jordan blocks, thus indicating stronger stability properties for the Hamiltonian system at hand. Under some quantitative hypotheses on $\cL$, we are able to obtain the following result.
}
\begin{proposition}\label{prop:orbstabgeneral} Let $\cL = \cJ \cH$ be a Hamiltonian operator as in Definition \ref{def:Hamiltonian}. Let $ {\rm X}= {\rm V}_s \oplus {\rm V}_u$, where $ {\rm V}_s$ and $ {\rm V}_u$ are orthogonal subspaces, invariant for $\cH$ with
\begin{equation}\label{def:b1b2}
\begin{aligned}
   b_1 &:= \inf \big\{ \langle \cH \omega , \omega \rangle :  \omega \in {\rm V}_s\,,\ \|\omega\|_{\rm X} = 1 \big\} \in (0, \infty) \,,\\
    b_2 &:= \sup \big\{ |\langle \cH \omega , \omega \rangle| :  \omega \in {\rm V}_u\,,\ \|\omega\|_{\rm X} = 1 \big\} \in (0, \infty) \,.
    \end{aligned}
\end{equation}
Let $P$ be the orthogonal projection onto the kernel of $\cJ$, and suppose that for some $b_3 \in (0,1]$ we have
\begin{equation}\label{def:b3}
    \| P \omega \|_{\rm X}^2 \geq b_3 \| \omega \|_{\rm X}^2\,, \quad \text{for every }\omega \in \rm{V}_u\, .
\end{equation}
Then, if 
\begin{equation}
    \label{bd:constats} {(1-b_3) b_2} < {b_3^2 b_1}\,,
\end{equation}
the operator $\cL$ is orbitally stable. More precisely, there exists a constant $C>0$ such that, for every $\omega_0\in\rm{X}$, the solution of the linear system $\de_t \omega(t) = \cL \omega(t)$ starting at $\omega(0) = \omega_0$ fulfills 
\begin{equation*}
\|\omega(t) \|_{\rm X} \leq C \| \omega_0 \|_{\rm X}\,,\quad \text{for every }t \in \RR\, .
\end{equation*} 
\end{proposition}
We note that the spaces ${\rm V}_s$ and ${\rm V}_u$ can be explicitly defined through the spectral projection theorem for the bounded self-adjoint operator $\cH$, meaning that condition \eqref{def:b1b2} is not restrictive. On the other hand, condition \eqref{def:b3} is in general nontrivial to verify.
We present two different proofs of the result. In the first one we explicitly construct a Lyapunov functional for the linear dynamics, whereas the second proof is based on spectral arguments.
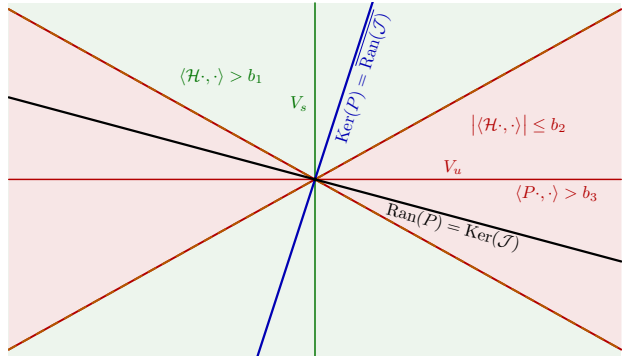
\begin{figure}[ht]
    \centering

\begin{tikzpicture}[
  x=1cm,y=1cm,
  line cap=round, line join=round,
  scale=0.65, transform shape
]
  \def\xmax{6.2}
  \def\ymax{3.6}
  \def\a{0.56}   
  \def\mB{3.1}   
  \def\mU{-0.27} 

  \colorlet{cGreen}{red!70!black}
  \colorlet{cRed}{green!45!black}

  \colorlet{cBlue}{blue!70!black}
  \colorlet{cBrown}{brown!70!black}
  \colorlet{cBound}{cGreen}
  \colorlet{cDash}{orange!70!black}

  \tikzset{
    ax/.style={line width=0.55pt},
    main/.style={line width=0.85pt},
    bound/.style={draw=cBound, line width=0.75pt},
    dashbound/.style={draw=cDash, line width=0.75pt, dashed},
    lab/.style={font=\small, inner sep=1.2pt},
    labG/.style={lab, text=cGreen},
    labR/.style={lab, text=cRed},
    labB/.style={lab, text=cBlue},
    labK/.style={lab, text=black},
  }

  \clip (-\xmax,-\ymax) rectangle (\xmax,\ymax);

  \fill[cGreen!10]
    (0,0) -- (\xmax, \a*\xmax) -- (\xmax,-\a*\xmax) -- cycle;
  \fill[cGreen!10]
    (0,0) -- (-\xmax,\a*\xmax) -- (-\xmax,-\a*\xmax) -- cycle;

  \begin{scope}
    \fill[cRed!7] (-\xmax,-\ymax) rectangle (\xmax,\ymax);
    \fill[cGreen!10]
      (0,0) -- (\xmax, \a*\xmax) -- (\xmax,-\a*\xmax) -- cycle;
    \fill[cGreen!10]
      (0,0) -- (-\xmax,\a*\xmax) -- (-\xmax,-\a*\xmax) -- cycle;
  \end{scope}

  \draw[ax, draw=cGreen] (-\xmax,0) -- (\xmax,0);
  \draw[ax, draw=cRed]   (0,-\ymax) -- (0,\ymax);

  \draw[bound] (-\xmax, \a*\xmax) -- (0,0) -- (\xmax,-\a*\xmax);
  \draw[bound] (-\xmax,-\a*\xmax) -- (0,0) -- (\xmax, \a*\xmax);

  \draw[dashbound] (-\xmax, \a*\xmax) -- (0,0) -- (\xmax,-\a*\xmax);
  \draw[dashbound] (-\xmax,-\a*\xmax) -- (0,0) -- (\xmax, \a*\xmax);

 
  \draw[main, draw=cBlue] (-\xmax,-\mB*\xmax) -- (\xmax,\mB*\xmax);

  \draw[main, draw=black] (-\xmax,-\mU*\xmax) -- (\xmax,\mU*\xmax);

  \fill[black] (0,0) circle (1.15pt);

  \node[labG, anchor=west] at ($(0,0)+(2.55,0.25)$) {$V_u$};
  \node[labG, anchor=west] at ($(0,0)+(3.1,1.1)$)
    {$\bigl|\langle \mathcal H\cdot,\cdot\rangle\bigr|\le b_2$};

  \node[labR, anchor=east] at ($(0,0)+(-0.1,1.55)$) {$V_s$};
  \node[labR, anchor=east] at ($(0,0)+(-1,2.10)$)
    {$\langle \mathcal H\cdot,\cdot\rangle>b_1$};

  \node[labG, anchor=west] at ($(0,0)+(4.00,-0.3)$)
    {$\langle P\cdot,\cdot\rangle>b_3$};

  \pgfmathsetmacro{\xB}{\ymax/\mB}

  \path (-\xB,-\ymax) -- (\xB,\ymax)
    node[labB, sloped, pos=0.8, below]
    {$\mathrm{Ker}(P)=\overline{\mathrm{Ran}(\mathcal{J})}$};

  \path (-\xmax,-\mU*\xmax) -- (\xmax,\mU*\xmax)
    node[labK, sloped, pos=0.73, below]
    {$\operatorname{Ran} (P)=\operatorname{Ker}(\mathcal J)$};

\end{tikzpicture}
\caption{Visual representation of the spaces $\mathrm{V}_s$, $\mathrm{V}_u$, $\mathrm{Ran}(P)$ and $\mathrm{Ker}(P)$ in Proposition \ref{prop:orbstabgeneral}, whose hypotheses ensure that the space $\overline{\Rn(\cJ)}$, represented by the blue line and containing all the possible growing modes, lies entirely in the stable cone $\langle \cH \cdot, \cdot\rangle > b_1 $. }
\end{figure}
\begin{proof} The proof is based on the conservation of the following bilinear form
\begin{equation*}
B(\omega) = \langle \cH \omega, \omega \rangle
+ \gamma \langle P \omega , P \omega \rangle
\end{equation*}
where we recall that $P$ is the projection to $\mathrm{Ker}(\cJ)$.
It is not hard  to see that if $\de_t\omega = \cJ \cH \omega$, then both $P\omega$ and $\langle \cH \omega, \omega \rangle$ are conserved under the linear flow, so $B(\omega)$ is conserved. Thus, the statement follows if we are able to show that $B(\omega) \gtrsim \| \omega \|^2$. We decompose $\omega = \omega_u + \omega_s$ with $\omega_u \in \rm{V}_u$ and $\langle  \omega_s, \omega_u \rangle = 0$. By \eqref{def:b1b2}, we bound 
\begin{align}\label{eqn:plugin}
    B(\omega_u+\omega_s)&\geq\, b_1\|\omega_s\|^2+\gamma \|P\omega_s\|^2+\gamma \|P\omega_u\|^2-b_2\|\omega_u\|^2-2\gamma|\langle P\omega_u,P\omega_s\rangle|
    .
\end{align}
Since $P^2=P$, and $\langle \omega_s,\omega_u\rangle=0$, we have that 
\begin{align}
    \langle P\omega_u,P\omega_s\rangle=\langle (P-\uno)\omega_u,\omega_s\rangle.
\end{align}
From \eqref{def:b3} and the fact that $\|\omega\|^2{=}\|P\omega\|^2+\|(\uno-P)\omega\|^2$, we have 
$   \|(\uno-P)\omega_u\|\leq {\sqrt{1-b_3}}\|\omega_u\|.$
Let $\alpha>0$ be a parameter to be defined later on and estimate, using also Young inequality,
\begin{align*}
    2\gamma |\langle P\omega_u,P\omega_s\rangle |&= 2(\gamma-\alpha) |\langle P\omega_u,P\omega_s\rangle|+2\alpha|\langle (\uno-P)\omega_u,\omega_s\rangle|
    \\&
    \leq 2|\gamma-\alpha|\| P\omega_u\|\|P\omega_s\|+2\alpha \sqrt{{1-b_3}}\|\omega_u\|\|\omega_s\|
    \\&\leq \gamma \|P\omega_s\|^2+
    \frac{(\gamma- \alpha)^2}{\gamma} \|P\omega_u\|^2+    
   \alpha^2\frac{1-b_3}{b_1(1-\gamma^{-1})} \|\omega_u\|^2
    +
    b_1(1-\gamma^{-1})\|\omega_s\|^2
    .
\end{align*}
Plugging it in \eqref{eqn:plugin} and using the lower bound $\|P\omega_u\|^2 \geq b_3 \|\omega_u\|^2$ in \eqref{def:b3} we get
\begin{align*}
    B(\omega_u+\omega_s)\geq\,& \frac{b_1}{\gamma}\|\omega_s\|^2+\Big(2\alpha b_3-b_2-\alpha^2\frac{1-b_3}{b_1(1-\gamma^{-1})}-\frac{\alpha^2b_3}{\gamma}\Big) \|\omega_u\|^2.
\end{align*}
Optimizing the choice of $\alpha$, we set $\alpha:=\dfrac{{b_1} {b_3} (\gamma-1) \gamma}{{b_1} {b_3} (\gamma-1)-({b_3}-1) \gamma^2}$. The coercivity of $B$ is given by
\begin{equation}
\frac{ \left({b_1} b_3^2-{b_2} (1-{b_3})\right){\gamma}^2-b_1 b_3 (b_2- {b_3}){\gamma} +{b_1} {b_2} {b_3}}{(1-{b_3}) {\gamma}^2+{b_1} {b_3} ({\gamma}-1)} >0
 .
\end{equation}
In view of \eqref{bd:constats}, the inequality above is true for $\gamma$ sufficiently large.
\end{proof}
\begin{proof}[Alternative proof] 
Let $(f,\lambda)$ be an eigenpair of $\cL$, namely $\cJ \cH f = \lambda f$ and let us suppose that $f$ is orbitally unstable, namely:
$$
\text{\it either}\quad ({\bf A})\ \ \Re\, \lambda \neq 0\quad \text{\it or}\quad ({\bf B})\ \ \Re\, \lambda = 0 \text{ and there exists } g\in {\rm X} \text{ such that }\cJ \cH g = \lambda g + f\,.
$$
Then, in both cases:
\begin{enumerate}
    \item $Pf = 0$, because $f$ belongs to the range of $\cJ$;
    \item $\langle \cH f ,f \rangle=0 $, by Lemma \ref{lemma:uniqueness_algebra} in case ({\bf A}) or, in case ({\bf B}),  because $\lambda =- \overline{\lambda}$ and
    $$
 \lambda \langle f, \cH g \rangle = \langle \cJ \cH f, \cH g \rangle = - \langle \cH f,\cJ  \cH g \rangle = - \overline{\lambda} \langle f, \cH g \rangle - \langle \cH f,  f \rangle\, .
$$
\end{enumerate}
 Let us decompose $f=u+s$, with $u \in \mathrm{V}_u$ and $s\in \mathrm{V}_s$. The two properties above give
$$
\mathrm{(i)}\ Pu = - Ps \, ,\quad \mathrm{(ii)}\ \langle \cH s,s\rangle = - \langle \cH u,u\rangle\,.
$$
From $\mathrm{(ii)}$ and \eqref{def:b1b2} we have that $b_1 \|s\|^2 \leq b_2 \|u\|^2 $, whereas from $\mathrm{(i)}$ and \eqref{def:b3} we obtain
$$
b_3 \|u\|^2 \leq \|Pu\|^2 = \langle Pu,u \rangle \stackrel{\mathrm{(i)}}{=} - \langle Ps,u \rangle = \langle  s, (\uno - P)u \rangle \leq \|s\| \|(\uno - P) u\| \leq  \sqrt{1-b_3} \|s\| \| u\|\,,
$$
where the last inequality descends from $\|Pu\|^2 \geq b_3 \|u\|^2 $. In the end we have
$$
b_1 \|s\|^2 \leq b_2 \|u\|^2\,,\quad b_3^2 \|u\|^2 \leq (1-b_3) \|s\|^2\, ,
$$
namely $\|s\|^2 \leq \frac{b_2(1-b_3)}{b_1 b_3^2} \|s\|^2  < \|s\|^2 $, by hypothesis. This is a contradiction.
\end{proof}

\section{Stable invariant subspaces of the Taylor-Green vortex} \label{sec:stability}
This section is devoted to detecting all the stable $\cL_E$-invariant subspaces of $L^2_0(\TT^2)$.
\subsection{Invariant subspaces of $\cL_E$} 
Let us consider the following splitting of $L^2_0(\TT^2)$ into
\begin{equation}\label{EvOddsplitting}
L^2_0(\TT^2) =  \mathrm{Odd}_x\mathrm{Odd}_y  \oplus   \mathrm{Odd}_x\mathrm{Ev}_y \oplus
  \mathrm{Ev}_x\mathrm{Odd}_y \oplus  \mathrm{Ev}_x\mathrm{Ev}_y
\,,
\end{equation}
where the first subspace is given by $L^2_0(\mathbb T^2)$ functions that are odd in both variables and the other subspaces are similarly defined. The operator $\cH$ in \eqref{cJcH} is invariant on the four spaces, since the Laplace operator keeps the parity of the function to which it is applied. The same holds for $\cJ$ that we can regard as a closed and densely-defined operator on each of the four subspaces. Indeed
\begin{equation}
    \{ \sin(x) \sin(y), g(x,y)\} \stackrel{\eqref{PoissonBracket}}{=} \cos(x) \sin(y) \de_y g(x,y) - \cos(y) \sin(x) \de_x g(x,y) \,, 
\end{equation}
and we readily see that the two terms on the right-hand side keep the parity of $g$ in both variables.
As a consequence, the operator $\cL_E$ in \eqref{cL} is invariant and Hamiltonian on the four subspaces in \eqref{EvOddsplitting}.

Moreover, we can find further invariant subspaces of each of them by considering the following additional splitting. Let $\mrA_x \mrB_y$, $\{\mrA,\mrB\} \subseteq \{\mathrm{Ev},\mathrm{Odd}\}$, be one of the four spaces in \eqref{EvOddsplitting}. Any function $g\in \mrA_x \mrB_y$ admits a representation in Fourier series as 
\begin{equation}\label{FourierRepresentation}
    g(x,y) = \sum_{j,k\in \NN} g_{j,k} \varsigma_{\mrA}(jx) \varsigma_{\mrB}(ky)\,,
\end{equation}
where  $\varsigma_{\mathrm{Ev}}(t) := \cos(t)$ and $\varsigma_{\mathrm{Odd}}(t) := \sin(t)$. The additional splitting is then given by
\begin{equation}\label{harmonicsplitting}
  \mrA_x \mrB_y = (\mrA_x \mrB_y )^{[\mathrm{ev}]} \oplus (\mrA_x \mrB_y )^{[\mathrm{odd}]}\,,
\end{equation}
with the two spaces respectively given by functions $g$ of the form \eqref{FourierRepresentation} where $g_{j,k}=0$ if $j+k$ is odd (for the first one) or even (for the second). Again we can readily observe that the two subspaces are invariant for the operator $\cH$ in \eqref{cJcH}. On the other hand, we observe that
\begin{align}\label{actionJonharmonics}
    &\big\{ \sin(x) \sin(y), \varsigma_\mrA(jx) \varsigma_\mrB(ky) \big\} \\ \notag &\quad\stackrel{\eqref{PoissonBracket}}{=} \cos(x)  \varsigma_\mrA(jx)\,  \sin(y) \big(\varsigma_\mrB(ky)\big)' -  \sin(x)\big(\varsigma_\mrA(jx)\big)'\,  \cos(y) \varsigma_\mrB(ky) \,,
\end{align}
resulting in a linear combination of the functions 
$$
\begin{aligned}
\varsigma_\mrA\big((j-1)x) \varsigma_\mrB\big((k-1)y)\,,\quad \varsigma_\mrA\big((j-1)x) \varsigma_\mrB\big((k+1)y)\,,\\
\varsigma_\mrA\big((j+1)x) \varsigma_\mrB\big((k-1)y)\,,\quad \varsigma_\mrA\big((j+1)x) \varsigma_\mrB\big((k+1)y)\,.
\end{aligned}
$$ 
In each of these four functions the sum of the activated harmonics has the same parity as $j+k$. The latter argument shows that also the operator $\cJ$ in \eqref{cJcH} is invariant on the two subspaces in the right-hand side of \eqref{harmonicsplitting}. Thus $\cL_E$ as a whole is invariant and Hamiltonian on the additional splitting.

Now, we can rule out the presence of instabilities on several of those subspaces by noticing the absence of negative directions of $\cH$.

\begin{lemma}\label{stability_easy_cases}
The spectrum of the operator $\cL_E $ in \eqref{cL} restricted to $(\mathrm{Odd}_x\mathrm{Odd}_y)^{[\mathrm{odd}]}$ and to $(\mathrm{A}_x\mathrm{B}_y)^{[\mathrm{even}]}$ for any $\mathrm{A}, \mathrm{B} \in \{ \mathrm{Ev}, \mathrm{Odd} \}$ is purely imaginary.
\end{lemma}
\begin{proof} 
The four eigenvectors of $\cH$ with negative eigenvalues were given in \eqref{negativedirectionscH}. We notice that none of them are in the subspaces considered above. Therefore, the result follows directly from Lemma \ref{lemma:uniqueness_algebra} with $k = 0$.
\end{proof}

We now turn our attention to the remaining $\cL_E$-invariant
subspaces where the stability properties cannot be trivially determined because of the presence of the negative directions of $\cH$ from \eqref{negativedirectionscH}.

\subsection{Stability of odd perturbations} \label{subsec:difficult_stability}

Let us consider the spaces $(\mathrm{Ev}_x\mathrm{Odd}_y)^{[\mathrm{odd}]}$ and $(\mathrm{Odd}_x\mathrm{Ev}_y)^{[\mathrm{odd}]}$ in \eqref{harmonicsplitting}. 
Let us first observe that the two spaces are conjugated by the linear involution 
\begin{equation}
    \label{sigmaoperator}
  \varrho : L^2_0(\TT^2) \to L^2_0(\TT^2)\,,\quad  \varrho f(x,y) := f(y,x)\,,\quad \varrho^2=\uno\, .
\end{equation}
In view of \eqref{cJcH}, one verifies that
\begin{equation}\label{signswitch}
    \varrho \cH = \cH \varrho\,,\quad \varrho \cJ = - \cJ \varrho \, .
\end{equation}
As consequence,
\begin{equation}\label{eqn:studiamounosolo}
    \cL_E|_{(\mathrm{Odd}_x\mathrm{Ev}_y)^{[\mathrm{odd}]}} = - \varrho \cL_E|_{(\mathrm{Ev}_x\mathrm{Odd}_y)^{[\mathrm{odd}]}} \varrho^{-1} \, ,
\end{equation}
hence the spectrum of $\cL_E$ is the same on  the two invariant subspaces.
By \eqref{negativedirectionscH}, the two subspaces contain each one negative direction of $\cH$. In this case, we prove stability as an application of our Theorem \ref{thm:instability_criterion}. To verify the criterion, we reduce the problem to a precise lower bound of $\| P  E_u \|_{L^2}^2$, (see \eqref{eq:final_condition_stability}), that in turn we verify by using properties of the period function in \eqref{eq:period_def}. Such properties are well studied and have been also used recently in \cite{brue2024enhanced} to prove mixing-estimates for the cellular flow.  
\begin{proposition}\label{prop:OddEvstability}
    The spectrum of the operator $\cL_E $ in \eqref{cL} restricted to each of its invariant subspaces 
    $(\mathrm{Ev}_x\mathrm{Odd}_y)^{[\mathrm{odd}]} $ and $(\mathrm{Odd}_x\mathrm{Ev}_y)^{[\mathrm{odd}]}$  in \eqref{harmonicsplitting} is purely imaginary.
\end{proposition}
\begin{proof}

\textbf{Step 1. Restriction to real eigenvalues and application of the stability criterion.}  In view of \eqref{eqn:studiamounosolo}, we restrict ourselves to the study of $(\mathrm{Odd}_x\mathrm{Ev}_y)^{[\mathrm{odd}]}$. 
Again, the latter space contains only one negative direction of $\cH$ in \eqref{cJcH}. This consideration alone justifies that
\begin{equation}
    \sigma_{(\mathrm{Odd}_x\mathrm{Ev}_y)^{[\mathrm{odd}]}}(\cL_E) \subseteq \RR \cup \im \RR
\end{equation}
Indeed, if a spectral value $\lambda$  is such that $\Re\, \lambda \neq 0$ (w.l.o.g.\ positive) and $\Im\, \lambda \neq 0$ then $\lambda$ would be an unstable eigenvalue and, by the reality of $\cL_E$, also $\bar \lambda$ would be a distinct unstable eigenvalue, which contradicts Lemma \ref{lemma:uniqueness_algebra}.

Now, we apply Theorem \ref{thm:instability_criterion}. In order to do that, we define $E_u = \sin(x)$, which spans the only negative direction of $\cH$ in $(\mathrm{Odd}_x\mathrm{Ev}_y)^{[\mathrm{odd}]}$. Then, according to Remark \ref{rem:n1phi0bis}, we set $s_u:=\frac35$ (which is the coercivity constant of $\cH$ on $\{E_u\}^\perp$ by \eqref{coercivewhole}), hence $h_u:=\frac85$ and we define the operators 
\begin{equation}
\cH_u f:= -\frac{8}{5} \langle E_u, f\rangle_{L^2} E_u , \qquad \cH_s := \cH - \cH_u.
\end{equation}
The hypotheses of Theorem \ref{thm:instability_criterion} are then clearly satisfied with $n = 1$, $s_u = \frac35$, $h_u = \frac85$. From \eqref{eq:achoice}, we readily deduce that
\begin{equation}
    \Phi(0)\geq \frac83\Big(\frac{\|PE_u\|^2}{\|E_u\|^2}-\frac58\Big).
\end{equation}
The proof of Proposition \ref{prop:OddEvstability} is then concluded if we can show that
\begin{equation} \label{eq:final_condition_stability}
\| P  E_u \|_{L^2}^2 > \frac58 \|E_u\|^2=\frac54 \pi^2.
\end{equation}

\textbf{Step 2. Computation of $\| P E_u \|_{L^2}$.} We consider the fundamental domain $\mathcal{C} = [0, \pi]^2$. We recall that the stream function of our Taylor-Green vortex is $\psi_E(x, y) = -\sin (x) \sin (y)$ and the operator $\cJ=-\{\psi_E,\cdot\}$. Then, for any $h \in (0, 1)$, the level set $\{ (x,y) \in \mathcal{C} : \sin(x)\sin(y) = h \}$ defines a closed streamline $\Gamma_h$, with normal vector ${\rm n} = (\cos (x) \sin(y), \sin (x) \cos( y) )$. We introduce the period function $T(h)$, defined as the derivative of the area enclosed by the streamline with respect to the energy $h$, or equivalently by the line integral:
\begin{equation} \label{eq:period_def}
    T(h) := \oint_{\Gamma_h} \frac{d\sigma}{|\nabla \psi_E|} = 4 \int_{\arcsin(h)}^{\pi/2} \frac{dx}{\sqrt{\sin^2(x) - h^2}}\,,
\end{equation}
where $d\sigma = \frac{| \nabla \psi_E | }{|\partial_y \psi_E |} dx $ denotes the arc-length measure. In the integral above, we notice that we are computing the contribution from $\Gamma_h \cap [0, \pi/2] \times [0, \pi/2]$ and multiplying by $4$ since the Taylor-Green vortex enjoys $90^\circ$-rotational symmetry around $(\pi/2, \pi/2)$. 

The projection $P$ onto the kernel of the transport operator $\cJ$ is the average over the streamlines $\Gamma_h$. For the unstable mode $E_u(x,y) = \sin(x)$, the projected function $P E_u$ is, remarkably, inversely proportional to the period, since by \eqref{eqn:projectionkerJ} we know that
\begin{equation} \label{eq:f0_bar_def}
    P {E}_u(h) = \frac{1}{T(h)} \oint_{\Gamma_h} E_u \frac{d\sigma}{|\nabla \psi|} = \frac{4}{T(h)} \int_{\arcsin(h)}^{\pi/2} \frac{\sin(x)}{\sqrt{\sin^2(x) - h^2}}\, dx  = \frac{2\pi}{T(h)}\,,
\end{equation}
where the last equality is justified substituting $u = \cos(x)$ to obtain that a primitive of the integrand is $-\arcsin\!\left(\frac{\cos x}{\sqrt{1-h^2}}\right)$.
{Now, we compute $T(h)$ explicitly. From \eqref{eq:period_def}, since $\sin^2(x)-h^2=(1-h^2)(1-\cos^2(x)/(1-h^2))$, we perform the change of variables $\cos(x)=\sqrt{1-h^2}\sin(s)$ to get that 
\begin{equation}
T(h)=4\int_0^{\pi/2}\frac{d s}{\sqrt{1-(1-h^2)\sin^2(s)}}.
\end{equation}
The integral above is proportional to the \emph{arithmetic-geometric mean} (AGM) iteration of Gauss, and from \cite[Theorem 1.1]{Borwein87} we know that
\begin{align}
\label{eq:AGM}
    \frac{1}{M(1,h)}=\frac{2}{\pi}\int_0^{\pi/2}\frac{ds}{\sqrt{1-(1-h^2)\sin^2(s)}}=\frac{T(h)}{2\pi},
\end{align}
where $M(1,h)$ is the AGM function\footnote{We recall that, given $0< b< a$, the function $M(a,b)$ is defined as follows. Let $a_{n+1}=(a_n+b_n)/2$ and $b_{n+1}=\sqrt{a_nb_n}$ with $a_1=a,b_1=b$. Then $M(a,b)=\lim_{n\to \infty}a_n=\lim_{n\to \infty}b_n$. Moreover, $b_n\leq M(a,b)\leq a_n$ for any $n\geq 1$.}. Then, to compute the norm of $PE_u$ we notice that the area in action-angle variables is given by $dx\,dy = T(h) dh\, d \theta$ (where $\theta$ is the normalized time along the streamline). Since  $\TT^2$ consists of four cells identical to $\mathcal{C}$, combining \eqref{eq:f0_bar_def} with \eqref{eq:AGM} we see that
\begin{equation} \label{eq:norm_explicit}
    \|P {E}_u\|_{L^2(\mathbb{T}^2)}^2 = 4 \int_{0}^{1}\int_0^1 |P {E}_u(h)|^2 T(h)\, dh d\theta = 8\pi\int_0^1M(1,h)dh,
\end{equation}
To bound the integral, we simply use the fact that, by the definition of $M(1,h)$, we know  $M(1,h)\geq \sqrt{h}$. Therefore 
\begin{align}\label{eqn:lowbdPEu}
    \|PE_u\|^2_{L^2(\TT^2)}\geq \frac{16}{3}\pi= \frac{8}{3\pi}\|E_u\|^2_{L^2(\TT^2)},
\end{align}
which proves the desired bound  \eqref{eq:final_condition_stability}.}
We finally comment that, numerically, $ \|P {E}_u\|^2/\|E_u\|^2\approx0.89$, and our lower bound gives $0.84$. This is a good approximation that suffices for our purposes.

\end{proof}

\begin{remark}[Orbital stability of the  Taylor-Green vortex in certain invariant spaces] 

Thanks to Proposition \ref{prop:orbstabgeneral} we can actually prove that on the space $(\mathrm{Odd}_x\mathrm{Ev}_y)^{[\mathrm{odd}]} $--and hence on the space $(\mathrm{Ev}_x\mathrm{Odd}_y)^{[\mathrm{odd}]} $ as well--the operator $\cL_E$ is orbitally stable.

 Indeed, keeping the notation introduced in Proposition~\ref{prop:orbstabgeneral}, we have that $b_2 = 1$ (greatest negative direction of $\cH$) and $b_1 = \frac35$ (coercivity constant of $\cH$ on the perpendicular modes). Thus, the condition \eqref{bd:constats} reads
$$ \frac{1-b_3}{b_3^2} < \frac35.$$
Regarding the value of $b_3$, we recall that \eqref{eqn:lowbdPEu} shows $\| P E_u \|^2 > \frac8{3\pi}\| E_u \|^2$, so that we one can take $b_3=\frac8{3\pi}$ and it is straightforward to check that condition \eqref{bd:constats} is satisfied.
\end{remark}

\section{Proof of instability}

 In Section \ref{sec:stability} we established the stability of $\cL_E$ in all spaces of the form $(\mrA_x \mrB_y )^{[\mrC]}$ with $\mrA, \mrB, \mrC \in \{ \mathrm{Ev},\mathrm{Odd} \}$, except for the case $(\mathrm{Ev}_x \mathrm{Ev}_y)^{[\mathrm{odd}]}$. Now, we turn our attention to this case, where we show the existence of a conjugate pair of unstable eigenfunctions of $\cL_E$. The main result of this section is the following.
\begin{proposition} \label{prop:EvEvstability} $\cL_E : (\mathrm{Ev}_x \mathrm{Ev}_y)^{[\mathrm{odd}]} \to (\mathrm{Ev}_x \mathrm{Ev}_y)^{[\mathrm{odd}]}$ has exactly one eigenvalue $\lambda_\star$ in the region $$ \lambda_\star \in [0.10, 0.17] - \im [0.57, 0.63]$$ Moreover, its
spectrum outside the imaginary axis is formed by the simple
eigenvalues
$\{ \lambda_\star, \bar \lambda_\star, -\lambda_\star, -\bar\lambda_\star \}$.
\end{proposition}
\subsection{
Complex invariant subspaces
}

 Let us introduce the auxiliary subspace $\rm{X}\subset (\mathrm{Ev}_x \mathrm{Ev}_y)^{[\mathrm{odd}]}  $ of functions $g$
admitting a Fourier series 
$
g(x,y) = \sum g_{j,k} \cos(j x) \cos(k y) 
$, with $g_{j,k}=0$ if $j$ is even and $k$ is odd. Notice that the orthogonal of $\rm{X}$ in $(\mathrm{Ev}_x \mathrm{Ev}_y)^{[\mathrm{odd}]}  $  is given by Fourier series with $g_{j,k}=0$ if $j$ is odd and $k$ is even.
In view of \eqref{cJcH} and
\eqref{actionJonharmonics}, one has $$ \cH\big(\rm{X}\big) = \rm{X}\,,\quad\cH\big(\rm{X}^\perp\big) = \rm{X}^\perp\,,\quad \cJ(\rm{X}),\;\cL_E(\rm{X} ) \subseteq \rm{X}^\perp\,, \quad \cJ(\rm{X}^\perp),\;\cL_E(\rm{X}^\perp) \subseteq \rm{X}\, .$$

The spaces $\rm{X}$ and $\rm{X}^\perp$ are alternating also for the involution
$\varrho$ in \eqref{sigmaoperator}. By \eqref{signswitch},
the spaces
\begin{equation}
    \label{EvEVpm}
   (\mathrm{Ev}_x \mathrm{Ev}_y)^{[\mathrm{odd}]}_+ := \big\{   g + \im \varrho g : g \in \rm{X}   \big\}\,,\quad (\mathrm{Ev}_x \mathrm{Ev}_y)^{[\mathrm{odd}]}_- := \big\{  g - \im \varrho g : g \in \rm{X}   \big\}
 \end{equation}
 are invariant for $\cH$, $\cJ$ and $\cL_E$, since, for instance, $\cL_E(g+\im \varrho g)= \cL_Eg- \im \varrho\cL_Eg = h + \im \varrho h $, with $h:=- \im \varrho\cL_Eg  $, and form an $L^2$-orthogonal splitting of $ (\mathrm{Ev}_x \mathrm{Ev}_y)^{[\mathrm{odd}]} $. 

\begin{remark}
The invariant spaces in \eqref{EvEVpm}
are not closed under complex conjugation
and hence the spectrum of the operator $\cL_E$  restricted to $(\mathrm{Ev}_x \mathrm{Ev}_y)^{[\mathrm{odd}]}_\pm$ is not symmetric with respect to the real axis. On the other hand, $f$ is an eigenfunction of $\cL_E$ in $(\mathrm{Ev}_x \mathrm{Ev}_y)^{[\mathrm{odd}]}_+$ if and only if $\overline{f}$ is an eigenfunction in $ (\mathrm{Ev}_x \mathrm{Ev}_y)^{[\mathrm{odd}]}_- $. We then restrict our study  to $\cL_E$ on $(\mathrm{Ev}_x \mathrm{Ev}_y)^{[\mathrm{odd}]}_+$ and find the pair of eigenvalues $\{\lambda_\star,-\overline{\lambda}_\star\} $ in Proposition \ref{prop:EvEvstability}. 
\end{remark}
Let us equip the space  $(\mathrm{Ev}_x \mathrm{Ev}_y)^{[\mathrm{odd}]}_+$ with a Fourier basis.
Let 
$j,k \in \ZZ$, and define 
\begin{align}\label{def:Ejk}
E_{j, k} &:= \cos(jx)\cos(ky) + \im (-1)^{j} \cos(kx) \cos(jy) \,. 
\end{align}
 One readily observes that
$
E_{j,-k} = E_{j,k}\,,\; E_{-j,k}=\overline{E}_{j,k}$ and, for $j+k$ odd, $ E_{k,j} = (-1)^k \im E_{j,k}
$.
The set $\{ E_{j,k}: (j,k) \in \cI\} $, where
\begin{equation}
    \label{specialFourierbasis}
\cI := \{(j,k) \in \ZZ^2: j,k\geq0\,,\; j+k\text{ odd}\,,\; k\leq j\}\, ,
\end{equation}
is an orthogonal basis for the space  $(\mathrm{Ev}_x \mathrm{Ev}_y)^{[\mathrm{odd}]}_+$.
We compute the action of $\cJ$ on such basis.
\begin{lemma} \label{lemma:J_computation} 
The operator $\cJ$ in \eqref{cJcH} acts on the functions $E_{j,k}$ in \eqref{def:Ejk} as follows:
\begin{enumerate}[label=\arabic*)]
\item Case $(j, k) = (1, 0)$:
\begin{equation}\label{actionE10}
     \cJ E_{1, 0} = \frac{\im}{2} E_{1, 0} - \frac12 E_{2, 1}.
\end{equation}
\item Case $(j, k) = (j, 0)$ for $j\geq 3$:
$$ \cJ E_{j, 0} = \frac{j}{2} E_{j-1, 1} - \frac{j}{2} E_{j+1, 1}.$$
\item Case $(j, k) = (k+1, k)$ for $k \geq 1$:
$$ \cJ E_{j, k} =  \frac14( - E_{j+1, k+1} + E_{j-1, k-1}) - \frac{2k+1}{4} E_{j+1, k-1} + \frac{(2k+1)\im (-1)^{k}}{4} E_{j, k}.$$
\item In any other case:
$$\cJ E_{j, k} = \frac{j-k}{4}(- E_{j+1, k+1} + E_{j-1, k-1}) + \frac{j+k}{4} ( -E_{j+1, k-1} +  E_{j-1, k+1} ).$$
\end{enumerate}
\end{lemma}

\begin{figure}[ht]
    \centering
    \begin{tikzpicture}[scale=0.75, >=stealth]
    \tikzset{
        mode/.style={circle, draw=black, fill=white, inner sep=0pt, minimum size=5pt},
        source/.style={circle, draw=red, thick, fill=red!20, inner sep=0pt, minimum size=7pt},
        link/.style={<->, thick, gray!70},
        linklgd/.style={->, thick, gray!70}, 
        evolvelgd/.style={->, dashed, thick, gray!70},
        evolve/.style={<->, dashed, thick, gray!70},
        selfloop/.style={->, thick, blue, looseness=10, min distance=8mm, out=150, in=75},
        boundary/.style={->, thick, red, bend right=20},
        rootloop/.style={->, thick, purple, looseness=10, min distance=8mm, out=-135, in=-45}
    }

    \draw[->] (0,0) -- (7.5,0) node[right] {$j$};
    \draw[->] (0,0) -- (0,5.5) node[above] {$k$};

    \node[source] (N-1-0) at (1,0) {};
    \node[mode]   (N-3-0) at (3,0) {};
    \node[mode]   (N-5-0) at (5,0) {};
    \node[mode]   (N-2-1) at (2,1) {};
    \node[mode]   (N-4-1) at (4,1) {};
    \node[mode]   (N-6-1) at (6,1) {};
    \node[mode]   (N-3-2) at (3,2) {};
    \node[mode]   (N-5-2) at (5,2) {};
    \node[mode]   (N-4-3) at (4,3) {};
    \node[mode]   (N-6-3) at (6,3) {};
    \node[mode]   (N-5-4) at (5,4) {};

    \draw[selfloop] (N-1-0) to (N-1-0);
    \draw[selfloop] (N-2-1) to (N-2-1);
    \draw[selfloop] (N-3-2) to (N-3-2);
    \draw[selfloop] (N-4-3) to (N-4-3);
    \draw[selfloop] (N-5-4) to (N-5-4);

    \draw[rootloop] (N-1-0) to (N-1-0);

    \draw[link] (N-1-0) -- (N-2-1);
    \draw[link] (N-3-0) -- (N-2-1); \draw[link] (N-3-0) -- (N-4-1);
    \draw[link] (N-5-0) -- (N-4-1); \draw[link] (N-5-0) -- (N-6-1);
    \draw[link] (N-2-1) -- (N-3-2);
    \draw[link] (N-4-1) -- (N-3-2); \draw[link] (N-4-1) -- (N-5-2);
    \draw[link] (N-6-1) -- (N-5-2);
    \draw[link] (N-3-2) -- (N-4-3);
    \draw[link] (N-5-2) -- (N-4-3); \draw[link] (N-5-2) -- (N-6-3);
    \draw[link] (N-4-3) -- (N-5-4);
    \draw[link] (N-6-3) -- (N-5-4);

    \draw[boundary] (N-1-0) to (N-2-1);
    \draw[boundary] (N-3-0) to (N-2-1); \draw[boundary] (N-3-0) to (N-4-1);
    \draw[boundary] (N-5-0) to (N-4-1); \draw[boundary] (N-5-0) to (N-6-1);

    \draw[evolve] (N-5-4) -- ++(0.7, 0.7);
    \draw[evolve] (N-6-3) -- ++(0.7, 0.7);
    \draw[evolve] (N-6-3) -- ++(0.7, -0.7);
    \draw[evolve] (N-6-1) -- ++(0.7, 0.7);
    \draw[evolve] (N-6-1) -- ++(0.7, -0.7);

    \node[anchor=north west, draw=gray!20, fill=white, rounded corners, inner sep=4pt] at (9.0, 5.5) {
        \scriptsize
        \begin{tabular}{@{} c l @{}}
            \multicolumn{2}{l}{\footnotesize \textbf{Spectral Coupling}} \\ [4pt]
            \textcolor{red}{$\bullet$} & Unstable Mode $(1,0)$ \\ [2pt]
            \tikz[baseline=-0.6ex]{\draw[linklgd] (0,0) -- (0.35,0);} & Standard Interaction \\ [2pt]
            \textcolor{blue}{$\circlearrowleft$} & Reflection across $k = j$ \\ [2pt]
            \textcolor{red}{$\curvearrowright$} & Reflection across $k = 0$ \\ [2pt]
            \textcolor{purple}{$\circlearrowright$} & Double reflection \\ [2pt]
            \tikz[baseline=-0.6ex]{\draw[evolvelgd] (0,0) -- (0.35,0);} & Pattern Evolution
        \end{tabular}
    };

\end{tikzpicture}
    \caption{Visualization of the $\cJ$ acting on the basis $\{ E_{j, k} \}_{(j, k) \in \cI}$ of our invariant subspace $(\mathrm{Ev}_x \mathrm{Ev}_y)^{[\mathrm{odd}]}_+$. The gridpoint $(j, k)$ represents $E_{j, k}$ and it has arrows to each component of $\cJ E_{j, k}$. Gray arrows correspond to the general case, while coloured arrows correspond to the special cases arising from reflections.}
    \label{fig:spectral_lattice}
\end{figure}
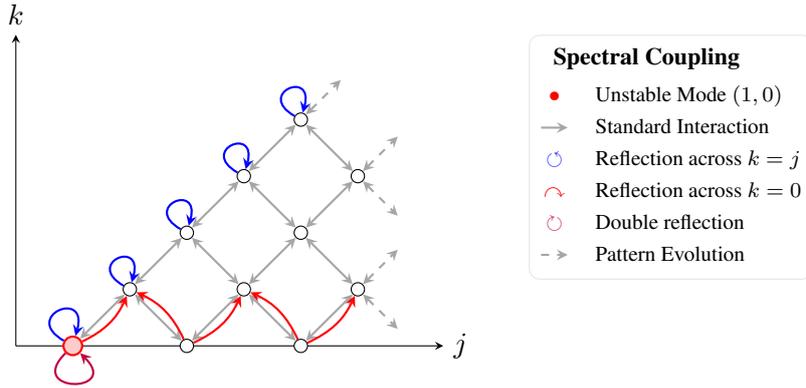

\begin{proof}

First recall that
 $   \cJ f =  \cos(x)\sin(y) \partial_y f - \sin(x)\cos(y) \partial_x f$
by \eqref{cJcH} and \eqref{PoissonBracket}.
We define
\begin{equation}
    \mathfrak{C}_{j,k} := \cos(jx)\cos(ky), \qquad \text{ so that } \qquad E_{j,k} = \mathfrak{C}_{j, k} + \im (-1)^j \mathfrak{C}_{k, j}.
\end{equation}
We start by computing the action of $\cJ$ on $\mathfrak{C}_{j, k}$
\begin{align}
    \cJ \mathfrak{C}_{j, k} &= \cos(x)\sin(y) \big[ -k \cos(jx)\sin(ky) \big] - \sin(x)\cos(y) \big[ -j \sin(jx)\cos(ky) \big] \nonumber \\
    &= -k [\cos(x)\cos(jx)] [\sin(y)\sin(ky)] \nonumber \\
    &\quad -\frac{k}{4} \big( \cos((j-1)x) + \cos((j+1)x) \big) \big( \cos((k-1)y) - \cos((k+1)y) \big) \nonumber \\
    &\quad  +j [\sin(x)\sin(jx)] [\cos(y)\cos(ky)] \nonumber \\
    &\quad +\frac{j}{4} \big( \cos((j-1)x) - \cos((j+1)x) \big) \big( \cos((k-1)y) + \cos((k+1)y) \big) \nonumber \\
    &=
    \frac{1}{4} \Big[ (j-k) \mathfrak{C}_{j-1, k-1} - (j-k) \mathfrak{C}_{j+1, k+1} + (j+k) \mathfrak{C}_{j-1, k+1} - (j+k) \mathfrak{C}_{j+1, k-1} \Big] \label{J_cosinus}
\end{align}
Now, we compute $\cJ E_{j, k}$ observing that $E_{j, k} = \mathfrak{C}_{j, k} + i(-1)^j \mathfrak{C}_{k, j}$, by  linearity:
\begin{equation} \label{eq:JE_1}
\cJ E_{j, k} = \frac{1}{4} \Big[ (j-k) E_{j-1, k-1} - (j-k) E_{j+1, k+1} + (j+k) E_{j-1, k+1} - (j+k)E_{j+1, k-1} \Big].
\end{equation}
However, the formula above exits our basis $\{ E_{j, k} \}_{(j, k) \in \mathcal I}$ whenever $k = 0$ or $j = k+1$, since it will propagate to functions of the type $E_{j, k}$ with $(j, k)\notin \mathcal I$. Therefore, we notice that
\begin{align} \begin{split} \label{eq:E_reflections}
E_{j, -1} &= \mathfrak{C}_{j, -1} + \im (-1)^j \mathfrak{C}_{-1, j} = \mathfrak{C}_{j, 1} + \im (-1)^j \mathfrak{C}_{1, j} = E_{j, 1}, \\
E_{k, k+1} &= \mathfrak{C}_{k, k+1} + \im(-1)^k \mathfrak{C}_{k+1, k} = \im(-1)^k \big[ \mathfrak{C}_{k+1, k} + \im(-1)^{k+1} \mathfrak{C}_{k, k+1} \big] = \im(-1)^k E_{k+1, k}.
\end{split} \end{align}
Combining \eqref{eq:JE_1} and \eqref{eq:E_reflections}, we get that for any $j \geq 3:$
\begin{align}\label{actionofJonE} \begin{split}
  \cJ E_{1,0} &= \frac{\im}{2}  E_{1,0} - \frac{1}{2} E_{2,1}, \\
     \cJ E_{j,0} &= \frac{j}{2} E_{j-1, 1} - \frac{j}{2} E_{j+1, 1}, \\
    \cJ E_{j, j-1} &= \frac{1}{4} E_{j-1, j-2} - \frac{1}{4} E_{j+1, j} - \frac{2j-1}{4} E_{j+1, j-2} + \im (-1)^k \frac{2j-1}{4} E_{j, j-1}.
\end{split} \end{align}
These formulas, together with \eqref{eq:JE_1} in the cases $j-1 > k > 0$, prove the Lemma.
\end{proof}

On the invariant subspace $(\mathrm{Ev}_x \mathrm{Ev}_y)^{[\mathrm{odd}]}_+$  the operator $\cH$ has only one negative direction, spanned by $E_{1,0}$. Consequently, by Lemma \ref{lemma:uniqueness_algebra}, the operator $\cL_E$ restricted to $(\mathrm{Ev}_x \mathrm{Ev}_y)^{[\mathrm{odd}]}_+$ has at most one simple unstable eigenvalue. We can localize the latter by means of the following lemma, which can be thought as an analogue of the classical Howard's semicircle theorem for shear flows \cite{drazin2004hydrodynamic}.
\begin{lemma} If $\cL_E|_{(\mathrm{Ev}_x \mathrm{Ev}_y)^{[\mathrm{odd}]}_+} $ has an unstable eigenvalue $\lambda$ then this lies in the semicircle 
\begin{equation}\label{semicircle}
    \Big\{ z \in \CC\;:\; \Re (z) >0 \,,\ \big|z + \tfrac\im2 \big| \leq \frac3{10}\sqrt{\frac56} \Big\}\, .
\end{equation}
\end{lemma}

\begin{proof} By \eqref{lemma:uniqueness_algebra}, we have $\langle \cH f, f\rangle = 0$ for every eigenfunction $f$ associated with $\lambda$. In particular, we can choose $f$ of the form
 $f = E_{1,0} + f_+$, with $f_+ \in E_{1,0}^\perp$. Then, observe that $\langle \cH E_{1,0}, E_{1,0} \rangle = - \| E_{1,0}\|_{L^2}^2 = -4\pi^2$  and $\langle \cH f_+, f_+ \rangle \geq \frac35 \|f_+\|_{L^2}^2$, by \eqref{coercivewhole}. Hence
\begin{equation}\label{Kreinestimate}
     0 = \langle \cH f, f\rangle =  \langle \cH E_{1,0}, E_{1,0} \rangle + \langle \cH f_+, f_+ \rangle \geq -4\pi^2+ \frac 3 5  \|f_+\|_{L^2}^2\, .
 \end{equation}
 In addition, the eigenvalue $\lambda$ is given by the inner product
\begin{equation}\label{auxinnerproduct}
    \langle \cL_E f , E_{1,0} \rangle = \lambda \langle  f , E_{1,0} \rangle  = \lambda \| E_{1,0}\|_{L^2}^2 = 4\pi^2\lambda \, .
\end{equation}
The left-hand side further develops into
\begin{align} \notag
    \langle \cL_E f , E_{1,0} \rangle &=  - \langle f, \cH \cJ E_{1,0} \rangle \stackrel{\eqref{actionE10}}{=}- \langle E_{1,0}+ f_+,  \frac{\im}{2} \cH E_{1,0} - \frac{1}{2} \cH E_{2,1} \rangle \\ \notag
    &=  \langle E_{1,0}+ f_+, \frac{\im}{2}  E_{1,0} + \frac{3}{10}  E_{2,1} \rangle = -\frac{\im}{2} (4\pi^2) + \frac{3}{10} \langle f_+ , E_{2,1} \rangle\, ,
\end{align}
which, combined with \eqref{auxinnerproduct} and since $\| E_{2,1} \|_{L^2} = \sqrt{2}\pi$, gives
\begin{equation}
    (4\pi)^2\big|\lambda + \tfrac{\im}{2} \big| \leq \frac{3}{10} \| f_+ \|_{L^2} \|E_{2,1} \|_{L^2}\stackrel{\eqref{Kreinestimate}}{\leq } (4\pi)^2\frac3{10}\sqrt{\frac56} \,.{\textcolor{white}{\qedhere}} \tag*{\qed}
\end{equation}
\end{proof}
\begin{remark} By relying on the Lin and Zeng  \cite{LinZengMemoirs} version of the Pontryagin invariant subspace theorem, one can ensure that the disk $|z+\frac\im2 | \leq \frac3{10} \sqrt{\frac56} $ contains two eigenvalues of $\cL$ either on the imaginary axis or forming a stable/unstable pair $\{ \lambda_*,-\bar\lambda_*\}$, with $\lambda_*$ lying in the semicircle \eqref{semicircle}. 
\end{remark}

\begin{figure}[h!!] \centering
\begin{tikzpicture}[>=stealth,scale=4]


  \draw[->] (-0.9,0) -- (0.9,0) node[right] {$\Im z$};
  \draw[->] (0,-0.2) -- (0,0.4) node[above] {$\Re z$};

  \coordinate (C) at (-0.5,0);

  \def\R{0.27} 

  \begin{scope}
    \clip (-1,0) rectangle (1,0.9); 
    \fill[blue!10] (C) circle (\R); 
    \draw[blue,thick] (C) circle (\R);

    \draw[blue,dashed] (C) -- ++(0,\R)
      node[midway,right] {$r$};
  \end{scope}

  \fill (C) circle (0.02);
  \node[below] at (C) {$-\tfrac{i}{2}$};

  \node[blue] at (-0.70,0.45) {$r \approx 0.27$};

  \coordinate (lam) at (-0.58,0.14);
  \fill[red] (lam) circle (0.02);
  \node[left] at (lam) {$\lambda_\star$};

  \coordinate (lam_sym) at (-0.58,-0.14);
  \fill[red] (lam_sym) circle (0.02);
  \node[left] at (lam_sym) {$-\overline{\lambda}_\star$};


\end{tikzpicture}
\begin{caption}{
Visual representation of the instability semicircle on the rotated $\lambda$-plane. The chosen value of $\lambda_\star$ is compatible with Proposition \ref{prop:EvEvstability}.}
\end{caption}
\end{figure}

\subsection{Application of Theorem \ref{thm:instability_criterion}}
We are now in a position to apply our stability criterion Theorem \ref{thm:instability_criterion} to $ \cL_E$ restricted to its invariant subspace $(\mathrm{Ev}_x \mathrm{Ev}_y)^{[\mathrm{odd}]}_+ \subset (\mathrm{Ev}_x \mathrm{Ev}_y)^{[\mathrm{odd}]}$. Let us define
\begin{equation} \label{eq:lesdiablerets}
E_u = E_{1, 0}, \qquad \cH_u(f) = -\frac85 E_u \langle E_u, f \rangle, \qquad \cH_s = \cH - \cH_u.
\end{equation}
 We point out that our definition of $E_u$ changes regarding the one used in Section \ref{subsec:difficult_stability}, since we are treating different subspaces of $L^2(\mathbb T^2)$. For any other mode $E_{j, k}$ with $(j, k)\in \mathcal I \setminus \{ (1, 0)\}$ we have 
 \begin{equation} \label{eq:formula_H}
 \mathcal H E_{j, k} = \mathcal H_s E_{j, k} = \left( 1 - \frac{2}{j^2+k^2} \right) E_{j, k}
 \end{equation}

We point out that our decomposition $\cH = \cH_s + \cH_u$ fulfills the hypothesis of Theorem \ref{thm:instability_criterion}, with $n=1$ and $s_u = \frac35$, $h_u = \frac85$. Hence, eigenvalues of $\cL$ are given by zeros of the holomorphic function
\begin{equation} \label{eq:new_Phi}
\Phi (\lambda) = 1 - \frac8{3\|E_u\|_{L^2}^2} \langle (\cT - \lambda)^{-1} \cT E_u, E_u \rangle\,.
\end{equation}

We will study $\Phi(\lambda)$ over a rectangular contour with corners given by:
\begin{equation} \label{eq:corners}
z_1 = \frac{10}{100} - \frac{57}{100} \im, \quad z_2 = \frac{17}{100} - \frac{57}{100} \im, \quad z_3 = \frac{17}{100} - \frac{63}{100} \im, \quad z_4 = \frac{10}{100} - \frac{63}{100} \im.
\end{equation}
We let $\Gamma$ to be the rectangular contour in the order: $z_1, z_2, z_3, z_4, z_1 $. Then, we have the following Lemma, which is proved by computer assistance by combining a rigorous winding argument and a rigorous primal-dual bound.

\begin{lemma} \label{lemma:holo_root} The holomorphic function $\Phi(\lambda)$ has exactly one simple root in the interior of $\Gamma$.
\end{lemma}

We defer the proof to the next subsection \ref{subsec:CAP} and in the remaining of this section, we proceed to prove Proposition \ref{prop:EvEvstability} assuming Lemma \ref{lemma:holo_root}.

\begin{proof}[Proof of Proposition \ref{prop:EvEvstability}]

The existence of $\lambda_\star$ being a simple eigenvalue of $\cL : (\mathrm{Ev}_x \mathrm{Ev}_y)^{[\mathrm{Odd}]}_+ \to (\mathrm{Ev}_x \mathrm{Ev}_y)^{[\mathrm{Odd}]}_+$ is directly guaranteed by the combination of Lemma \ref{lemma:holo_root} and Theorem \ref{thm:instability_criterion}. Moreover, by Lemma \ref{lemma:uniqueness_algebra} and $\mathrm{dim}({\rm V}_u) = 1$, we obtain that the simple eigenvalue $\lambda_\star$ is the only eigenvalue with positive real part in $(\mathrm{Ev}_x \mathrm{Ev}_y)^{[\mathrm{Odd}]}_+$. Then, we notice that conjugation sends $(\mathrm{Ev}_x \mathrm{Ev}_y)^{[\mathrm{Odd}]}_+$ to $(\mathrm{Ev}_x \mathrm{Ev}_y)^{[\mathrm{Odd}]}_-$ and vice versa, and moreover $\cL$ commutes with conjugation. Hence, the only eigenvalue of $\cL : (\mathrm{Ev}_x \mathrm{Ev}_y)^{[\mathrm{Odd}]}_- \to (\mathrm{Ev}_x \mathrm{Ev}_y)^{[\mathrm{Odd}]}_-$ with positive real part is $\bar \lambda_\star$, and moreover it is simple.

Thus, we have showed that $\cL : (\mathrm{Ev}_x \mathrm{Ev}_y)^{[\mathrm{Odd}]} \to (\mathrm{Ev}_x \mathrm{Ev}_y)^{[\mathrm{Odd}]}$ has two eigenvalues with positive real part: $\lambda_\star, \bar \lambda_\star$, which moreover are simple. From the decomposition $\cL = \cJ \cH$, we notice $\cL^\ast = \cH^\ast \cJ^\ast = -\cH \cJ = -\cH \cL \cH^{-1}$. Since $\cH$ is invertible (it is diagonal with non-zero entries), $\cL^\ast$ is similar to $-\cL$. This proves that the spectrum of $\cL$ is symmetric along the imaginary axis, concluding our proof. 
\end{proof}

\subsection{Computer-assisted winding argument} \label{subsec:CAP}

Our only remaining task is to prove Lemma \ref{lemma:holo_root}. The proof is based on a rigorous winding argument. The essential step is to rigorously enclose the inner product $\langle (\cT - \lambda)^{-1} \cT E_u, E_u \rangle$, needed to compute $\Phi(\lambda)$. We point out that performing rigorous computations of $\cT$ is trivial in our basis $\{ E_{j, k} \}_{(j, k) \in \cI}$, since the operator $\cT$ just propagate to neighbouring modes (see Lemma \ref{lemma:J_computation}). However, rigorously computing the resolvent $(\cT - \lambda )^{-1}$ is much more challenging. In fact, we only need to compute an inner product involving the resolvent and this can be done via the following.
\begin{lemma}[Primal-Dual Error Bound]\label{lem:primal_dual}
Let $\cT$ be a skew-adjoint linear operator on a Hilbert space ${\rm X}$, and for some $\lambda$ with $\text{Re}(\lambda )> 0$, consider its resolvent $\cR = (\cT - \lambda)^{-1}$. Let $v, v^\ast \in {\rm X}$. Let $u_{0}, u^\ast_{0} \in {\rm X}$. Let $r = v - (\cT-\lambda) u_{0}$ and $r^\ast = v^\ast - (\cT-\lambda)^\ast u^\ast_{0}$ be the respective residuals. Then:
\begin{equation}
    \left| \langle \cR v, v^* \rangle_{\rm X} - \left( \langle u_{0}, v^\ast \rangle_{\rm X} + \langle r, u^\ast_{0} \rangle_{\rm X} \right) \right| \leq \frac{ \| r \|_{\rm X}   \| r^* \|_{\rm X} }{| \mathrm{Re}(\lambda )|} \,.
\end{equation}
\end{lemma}
\begin{remark}
Even though the Lemma makes no assumptions on $u_0, u_0^\ast$, one should think of them as approximations of $\cR v$ and $\cR^{\ast}v^\ast$ respectively. The terms $\langle u_0, v^\ast \rangle$ and $\langle r, u^\ast_0 \rangle$ should be thought as terms that we can easily compute rigorously, since they only involve inner products and $\cT$, but not any inverse. The bound $ \| r \|_{\rm X} \, \| r^\ast \|_{\rm X}/|\Re(\lambda)|$ is an error term that we will rigorously account for in our code.
\end{remark}

\begin{proof}
We have
\begin{align*}
    \langle \cR v, v^\ast \rangle_{\rm X}
    &= \langle u_0 + \cR r, v^\ast \rangle_{\rm X} \\
    &= \langle u_0, v^\ast \rangle_{\rm X} + \langle \cR r, r^\ast+(\cT-\lambda)^\ast u^\ast_0 \rangle_{\rm X} \\
    &= \langle u_0, v^\ast \rangle_{\rm X} + \langle r, u^\ast_0 \rangle_{\rm X} + \langle \cR r, r^\ast \rangle_{\rm X}.
\end{align*}
The last term is trivially bounded by $\| \cR r \|_{\rm X} \| r^\ast \|_{\rm X} $, and  since $\cT$ is skew-adjoint, we clearly have $\| \cR \|_{{\rm X}\to {\rm X}} \leq 1/| \mathrm{Re} (\lambda )|$.
\end{proof}

Lemma \ref{lem:primal_dual} will allow us to rigorously enclose the set $\Phi (\lambda_I)$ in a computable complex rectangular region, being $\lambda_I$ another complex rectangular region. In order to compute a winding number on $\Phi$ with such machinery, we will use the following Lemma. 

\begin{lemma}[Rigorous Winding Number]\label{lem:winding_number}
Let $\Phi(\lambda)$ be an holomorphic function and $\Gamma \subset \mathbb{C}$ be a non self-intersecting contour in the complex plane. Assume that there exists a family of sets $\{ \Gamma_i \}_{i=0}^N$ covering $\Gamma \subset \cup_{i=0}^N \Gamma_i$ such that $\Gamma_0 = \Gamma_N$ and $\Gamma \cap \Gamma_i \cap \Gamma_{i+1} \neq \emptyset$ for $i = 0, 1, \ldots N-1$. Moreover, assume that there exists a family of angles $\{\theta_j\}_{j=0}^N$ with $e^{i\theta_0}=e^{i\theta_N}$ be such that: 
\begin{enumerate}
\item The set $\Phi (\Gamma_j)$ is $\theta_j$-oriented, meaning $\Phi (\lambda) e^{-i\theta_j}$ has strictly positive real part for all $z\in \Gamma_i$. 
\item The angles $\theta_i$ are locally consistent, meaning $| \theta_i - \theta_{i+1}| < \pi$ for all $i = 0, 1, 2, \ldots N-1$.
\end{enumerate}
Then, $\theta_N - \theta_0 = 2n\pi$, where $n$ is the number of roots of $\Phi (\lambda ) $ in the interior of $\Gamma$.
\end{lemma}

\begin{remark} \label{rem:winding}
Once values of $\theta_i'$ are found ensuring condition (1), it is trivial to ensure (2) by iteratively considering $\theta_i = 2\pi m_i + \theta_i'$ for some $m_i \in \mathbb Z$ such that $| \theta_i - \theta_{i-1} | \leq \pi$, for any $i \geq 1$. The possibility $| \theta_i - \theta_{i-1}| = \pi$ can be discarded using $\Phi (\Gamma_i) \cap \Phi (\Gamma_{i-1}) \neq \emptyset$, ensuring (2) is satisfied.
\end{remark}
\begin{proof}
 We partition the contour $\Gamma$ by choosing intersection points $w_j \in \Gamma_j \cap \Gamma_{j+1} \cap \Gamma$ for $j=0, \dots, N-1$. Let $w_N = w_0$. We compute $n$ with the standard complex analysis argument principle, and we decompose the contour integral into a sum of integrals over the arcs $\gamma_{j+1}$ connecting $w_j$ to $w_{j+1}$. Thus:
\begin{equation} \label{eq:winding}
    n = \frac{1}{2\pi i} \oint_\Gamma \frac{\Phi'(\lambda)}{\Phi(\lambda)} d\lambda = \frac{1}{2\pi i} \sum_{j=0}^{N-1} \int_{w_j}^{w_{j+1}} \frac{\Phi'(\lambda)}{\Phi(\lambda)} \, d\lambda.
\end{equation} 
The arc $\gamma_{j+1}$ is contained entirely within $\Gamma_{j+1}$, and since $\Phi(\Gamma_k)$ is $\theta_k$-oriented, the image $\Phi(\Gamma_k)$ lies in the half-plane $H_k = \{ w \in \mathbb{C} : \Re(w e^{-i\theta_k}) > 0 \}$. On this domain, we define a specific branch of the logarithm, $\log_k : H_k \to \mathbb{C}$, given by
\begin{equation}
    \log_k(w) = \mathrm{Log}(w e^{-i\theta_k}) + \im\theta_k,
\end{equation}
where $\mathrm{Log}$ denotes the principal branch of the logarithm. Hence
\begin{align*}
    \oint_\Gamma \frac{\Phi'(\lambda)}{\Phi (\lambda)}d\lambda &= \sum_{j=0}^{N-1} \left[ \log_{j+1}(\Phi(w_{j+1})) - \log_{j+1}(\Phi(w_j)) \right] \\
    &= \im (\theta_N - \theta_0) + \sum_{j=0}^{N-1} \left[ \log_j(\Phi(w_j)) - \log_{j+1}(\Phi(w_j)) \right], 
\end{align*}
where in the last identity we used that $\log_N (\Phi(w_N)) - \log_0 (\Phi(w_0)) = \im (\theta_N - \theta_0)$ since $e^{\im \theta_N} = e^{\im\theta_0}$ and $w_N = w_0$.

Finally, combining this with \eqref{eq:winding}, we just need to show that all terms into the right sum are zero. Take $w_j$ for $j=1, \dots, N-1$, and let $\phi = \Phi(w_j)$. Since $w_j \in \Gamma_j \cap \Gamma_{j+1}$, $\phi$ is in the domain of both $\log_j$ and $\log_{j+1}$.
\begin{align*}
    \log_j(\phi) - \log_{j+1}(\phi) &= \left( \mathrm{Log}(\phi e^{-\im\theta_j}) + \im\theta_j \right) - \left( \mathrm{Log}(\phi e^{-\im\theta_{j+1}}) + \im\theta_{j+1} \right) \\
    &= \mathrm{Log}(\phi e^{-\im\theta_j}) - \mathrm{Log}(\phi e^{-\im\theta_{j+1}}) + \im(\theta_j - \theta_{j+1}) \\
    &= \mathrm{Log} \left(\frac{\phi e^{-\im\theta_j}}{\phi e^{-\im\theta_{j+1}}} \right) + \im(\theta_j - \theta_{j+1}) = \mathrm{Log} ( e^{\im(\theta_{j+1} - \theta_j )} ) + \im (\theta_j - \theta_{j+1}) = 0
\end{align*}
In the third line we used in an essential way that the arguments of both $\Log$ functions have strictly positive real part (from the $\theta_i$-oriented assumption) so that $\Log (A) - \Log (B) = \Log (A/B)$ is valid. In the last equality, we also used the compatibility condition $|\theta_j - \theta_{j+1}| < \pi$ to ensure $\Log (e^{\im (\theta_{j+1} - \theta_j )} ) = \im(\theta_{j+1} - \theta_j)$.
\end{proof}

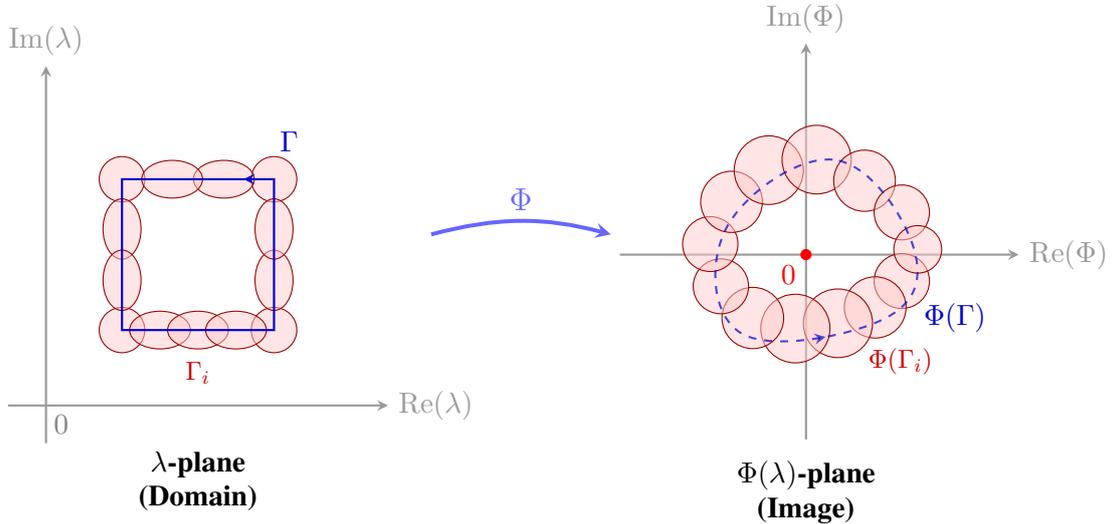
\begin{figure}[ht]
    \centering
    
    \begin{tikzpicture}[>=stealth, scale=1.0]

    \tikzset{
        complexaxis/.style={->, thick, gray!80},
        contour/.style={thick, blue!80!black, decoration={markings, mark=at position 0.55 with {\arrow{>}}}, postaction={decorate}},
        mapping/.style={->, line width=1.5pt, blue!60, shorten >= 2pt, shorten <= 2pt},
        covering/.style={fill=red!20, draw=red!60!black, thin, fill opacity=0.5}
    }

  \begin{scope}[local bounding box=leftscope]
        \draw[complexaxis] (-0.5,0) -- (4.5,0) node[right] {$\mathrm{Re}(\lambda)$};
        \draw[complexaxis] (0,-0.5) -- (0,4.5) node[above] {$\mathrm{Im}(\lambda)$};
        \node[gray] at (0.2, -0.25) {$0$};

        \draw[covering] (3, 1) circle (0.3); 
        \draw[covering] (3, 3) circle (0.3); 
        \draw[covering] (1, 3) circle (0.3); 
        \draw[covering] (1, 1) circle (0.3); 

        
        \foreach \x in {1.5, 2.0, 2.5} {
            \draw[covering] (\x, 1) ellipse (0.4 and 0.25);
        }
        \node[red!80!black, below, scale=0.9] at (2.0, 0.7) {$\Gamma_i$};

        \foreach \y in {1.66, 2.34} { 
            \draw[covering] (3, \y) ellipse (0.25 and 0.4); 
        }

        \foreach \x in {2.34, 1.66} { 
            \draw[covering] (\x, 3) ellipse (0.4 and 0.25); 
        }

        \foreach \y in {2.34, 1.66} { 
            \draw[covering] (1, \y) ellipse (0.25 and 0.4); 
        }

        \coordinate (A) at (1,1);
        \coordinate (B) at (3,1);
        \coordinate (C) at (3,3);
        \coordinate (D) at (1,3);
        \draw[contour] (A) -- (B) -- (C) -- (D) -- cycle;
        \node[blue!80!black, font=\bfseries] at (3.2, 3.5) {$\Gamma$};

        \node[align=center, font=\bfseries] at (2.0, -1) {$\lambda$-plane\\(Domain)};
    \end{scope}

    \draw[mapping, bend left=15] (5, 2.25) to node[above, font=\large] {$\Phi$} (7.5, 2.25);

    \begin{scope}[shift={(10,2)}, scale=0.7, local bounding box=rightscope]
        
        \draw[complexaxis] (-3.5,0) -- (4.0,0) node[right] {$\mathrm{Re}(\Phi)$};
        \draw[complexaxis] (0,-3.5) -- (0,4.0) node[above] {$\mathrm{Im}(\Phi)$};
        
        \fill[red] (0,0) circle (3pt); 
        \node[red, below left] at (0,0) {$0$};

        \draw[covering] (1.8, -0.5) circle (0.52);
        \draw[covering] (1.3, -1.0) circle (0.58);
        \node[red!80!black, below right, scale=0.9] at (1.0, -1.6) {$\Phi(\Gamma_i)$};
        \draw[covering] (0.6, -1.3) circle (0.65);
        \draw[covering] (-0.2, -1.4) circle (0.65);
        \draw[covering] (-1.0, -1.2) circle (0.58);
        \draw[covering] (-1.6, -0.6) circle (0.52);
        \draw[covering] (-1.8, 0.2) circle (0.52);
        \draw[covering] (-1.4, 1.0) circle (0.58);
        \draw[covering] (-0.7, 1.6) circle (0.65);
        \draw[covering] (0.2, 1.8) circle (0.65);
        \draw[covering] (1.1, 1.4) circle (0.58);
        \draw[covering] (1.8, 0.8) circle (0.52);
        \draw[covering] (2.1, 0.1) circle (0.45);

        \draw[contour, dashed, opacity=0.7] plot [smooth cycle, tension=0.7] coordinates {
            (2.0, 0.2) (0.5, 1.8) (-1.5, 0.5) (-1.2, -1.5) (1.5, -1.2)
        };
        \node[blue!80!black] at (2.8, -1.2) {$\Phi(\Gamma)$};

        \node[align=center, font=\bfseries] at (0, -4.5) {$\Phi(\lambda)$-plane\\(Image)};
    \end{scope}

\end{tikzpicture} 
    \caption{
    Visual representation of the winding number argument. The left panel shows the covering of the contour $\Gamma$, with each pair of consecutive $\Gamma_i, \Gamma_{i+1}$ intersecting over the contour $\Gamma$. The right panel illustrates a rigorous enclosure of $\Phi(\Gamma_i)$ . The small circular regions represent the image sets $\Phi(\Gamma_i)$, which satisfy the $\theta_i$-oriented condition (lying in a half-plane away from the origin). In our computer-assisted application, the enclosures will be balls, and the $\theta_i$ will point to the center of those balls, although this is not necessary for the validity of the Lemma.}
\end{figure}

\begin{proof}[Proof of Lemma \ref{lemma:holo_root}] The proof is computer-assisted and the code TaylorGreen\_CAP.ipynb can be found in the supplementary material (downloading the arXiv submission of this paper). The code was written and executed in Sage (\textit{SageMath-10.7}). It took 20 seconds to run in a personal laptop. The code consists of three notebook cells implementing the argument, which we describe below.

The first cell constructs a finite truncation of the matrices $\cJ, \cH_s, \cH_u, \cT$ over the base $E_{j, k}$. We also construct a weight vector $\cW$ encoding $\| E_{j, k} \|_{L^2}^2$, so that $\langle u, v \rangle_{L^2} = \sum_{(j, k)\in \cI} u_{j, k} v_{j, k} \cW_{j, k}$ being $u_{j, k}, v_{j, k}$ the $E_{j, k}$ coefficients of $u$ and $v$. We consider $\cI_t$ to be the subset of $\cI$ formed by $j+k \leq 160$, and map $\cI_t$ to a standard array of consecutive natural numbers. The matrices $\cH_s, \cH_u$ are implemented via \eqref{eq:lesdiablerets}, $\cJ$ via Lemma \ref{lemma:J_computation}, and $\cT$ as $\cH_s^{1/2} \cJ \cH_s^{1/2}$. We point out that this allows to rigorously compute objects such as $\cH_u v$, $\cH_s v$, $\langle u, v \rangle_{L^2}$, etc. as long as one works with interval arithmetics operations, which are guaranteed to rigorously enclose the result (see subsection \ref{subsec:intro_literature} for a high-level description of interval arithmetics and related literature). However, in the cases of $\cJ$ and $\cT$, rigorous computation is only guaranteed if $v$ is supported only on modes with $j+k \leq 160-2 = 158$, since otherwise there would be a component of the result at modes outside $\cI_t$. Thus, we define the set $\cI_{tt}$ to be the set of trusted indices: $\cI_{tt} = \{ (j, k)\in \cI_t \; : \; j+k \leq 158 \}$, and note that $\cJ u, \cT u$ can be computed rigorously as long as $u \in \text{span} \{ E_{j, k} \}_{(j, k) \in \cI_{tt}}$. The implementations are based on Real Ball Fields and Complex Ball Fields, that allow for rigorous, interval-arithmetic computations. We also have an additional numerical (non-rigorous) implementation of $\cT$ in numpy.scipy that will be used to obtain the approximations needed to use Lemma \ref{lem:primal_dual}. Indeed, since Lemma \ref{lem:primal_dual} works for any $u_0, u_0^\ast$, it is enough to obtain these candidates with a (non rigorous) numerical approximation of the resolvent with $\cT$ obtained in numpy.scipy. 

In the second cell we implement a rigorous computation of a enclosure for $\Phi(\lambda^I)$, being $\lambda^I$ a complex interval. This is based on Lemma \ref{lem:primal_dual}. The function ``solve\_trusted\_least\_squares" yields a numerical approximation $u \approx (\cT - \lambda)^{-1} (v)$ which is supported on the trusted modes $\cI_{tt}$, obtained by minimizing $\| (\cT - \lambda)^{-1} v \|_{L^2}^2$ over $\text{span} \{ E_{j, k} \}_{(j, k) \in \cI_{tt}}$. Since that approximation is not rigorous, we take $\lambda$ to be the center of the interval $\lambda_I$. Then, in ``compute\_primal\_dual\_rigorous", such numerical approximations are enclosed using complex intervals in order to rigorously treat the main terms and error terms of Lemma \ref{lem:primal_dual}. Then, in ``Phi" we rigorously enclose $\Phi(\lambda_I)$ via formula \eqref{eq:new_Phi}.

Lastly, we perform the winding number computation of $\Phi$ along $\Gamma$ via Lemma \ref{lem:winding_number}. Concretely, we define 26 points $w_i \in \Gamma$ (which contains the corners $z_i$), with $w_0 = w_{26}$ and define $\Gamma_i$ to be a complex interval containing both $w_i$ and $w_{i+1}$. In particular, since the part of $\Gamma$ between $w_i$ and $w_{i+1}$ is a straight segment, it will be contained in $\Gamma_i$, and since $w_{i+1} \in \Gamma_i \cap \Gamma_{i+1} \cap \Gamma$, we satisfy the corresponding hypothesis from Lemma \ref{lem:winding_number}. In order to check that $\Phi (\Gamma_i)$ is $\theta_i$-oriented, we use our rigorous enclosure for $\Phi(\Gamma_i)$. Our $\theta_i$ do not necessarily satisfy $| \theta_i - \theta_j | < \pi$ but this is fixed a posteriori by adding multiples of $2\pi$ as specified in Remark \ref{rem:winding}. Finally, we compute $\frac{\theta_{26} - \theta_0}{2\pi}$, and we obtain it is rigorously enclosed in $1 \pm 2.38\cdot 10^{-60}$. Since $\Phi$ is holomorphic this must be an integer, and therefore it is $1$, concluding the proof of Lemma \ref{lemma:holo_root}.

\end{proof}

\subsection{Proof of Theorem \ref{linthm}}

We recall the decompositions of $L_0^2(\mathbb T^2)$ into invariant spaces of $\cL_E$ due to \eqref{EvOddsplitting} and \eqref{harmonicsplitting}. Due to Lemma \ref{stability_easy_cases} and Proposition \ref{prop:OddEvstability}, we know that the spectrum of $\cL_E$ restricted to any of those invariant subspaces is purely imaginary, with the exception of $(\mathrm{Ev}_x \mathrm{Ev}_y )^{[\mathrm{odd}]}$. In that subspace, Proposition \ref{prop:EvEvstability} shows that there exist exactly four simple eigenvalues outside the imaginary axis, given by $\{ \lambda_\star, \bar \lambda_\star, -\lambda_\star, -\bar \lambda_\star \}$. 
\begin{remark}[On the regularity of the eigenfunction]\label{rmk:nonsmooth}

The eigenfunction $\omega_\star$ in Theorem \ref{linthm}, with $\cL_E \omega_\star = \lambda_\star \omega_\star$, is expected to be smooth on any superlevel set of the streamfunction $\{\psi_E > h\}$ (resp. $\{\psi_E <- h\}$) for any $h>0$, whereas it should have limited regularity at the hyperbolic point, more precisely we expect that $\omega_\star\notin H^{1-\Re \lambda_\star}$ in any neighborhood of the hyperbolic point.

To see this, we rewrite the eigenvalue equation as $( \cJ - \lambda_\star ) \cH \omega_\star = - 2 \lambda_\star \Delta^{-1} \omega_\star
$, and, using the formula for the resolvent of $\cJ - \lambda_\star $ in terms of the smooth flow $X$ of the Taylor-Green vortex,
\begin{equation}
    \label{eqn:invert-transp}
    \cH \omega_\star= 2 \lambda_\star\int_0^\infty e^{-s\lambda_\star}  \Delta^{-1} \omega_\star\circ X(s) \, ds.
\end{equation}
When restricting to any superlevel set $\{\psi_E > h\}$ for any $h>0$,  the right hand side gains two derivatives (to fix the ideas, at first  $\omega_\star \in L^2$ and $\Delta^{-1} \omega_\star\circ X \in H^2$), and this regularity is transferred to $ \cH \omega_\star$ through \eqref{eqn:invert-transp}, and therefore to $\omega_\star$ by elliptic regularity. {Note that this is possible because the flow lines are periodic orbits with smooth and bounded period (with estimates degenerating as $h\to 0$), meaning that each derivative in \eqref{eqn:invert-transp} causes a loss of order $\mathcal{O}(s)$, which can be easily absorbed by the exponential decay.} This argument can be bootstrapped to get smoothness.

The irregularity of $\cH \omega_\star $, and therefore of $ \omega_\star $, at the origin can be guessed heuristically by noticing that the eigenvalue equation rewrites as 
\begin{equation}\label{eq:usingfE}
(\cJ - \lambda_\star ) (\cH \omega_\star 
-
2\Delta^{-1} \omega_\star (0, 0))= -2\lambda_\star \Delta^{-1}\omega_\star 
+2\lambda_\star \Delta^{-1} \omega_\star (0, 0)\,.
\end{equation}
and therefore along the axis $x = 0$ simplifies to
\begin{equation}\label{eq:eigforx=0}
(\sin (y) \partial_y - \lambda_\star ) f(0, y) = E(0, y)
\end{equation}
where 
$$f := \cH \omega_\star (x,y)
-
2\Delta^{-1} \omega_\star (0, 0)\,,\quad E := -2\lambda_\star \Delta^{-1}\omega_\star (x,y)
+2\lambda_\star \Delta^{-1} \omega_\star (0, 0).$$ 
Here, by \eqref{eqn:somereg} we have that $\omega_\star \in H^{1/5}$; hence $E \in H^{2+1/5}$, vanishes at $(0, 0)$ and is even in both variables, which implies that $|E(x, y)| \lesssim (|x|+|y|)^{1+1/5}$. 
\\

Let $M(y) := \exp \int_{\pi/2}^y \frac{-\lambda_\star}{\sin (y')} dy' = |\tan (y/2)|^{-\lambda_\star}$ and observe that $M'(y)= - \frac{\lambda_\star}{\sin(y)}M(y)$.
Multiplying the equation by $M(y)$ and integrating, since $M(\pi/2)=1$, 
\begin{equation} \label{eq:formula_G}
M(y) f(0,y) = f(\pi/2) 
-\int_{\pi/2}^y \frac{E(0,y') M(y')}{\sin (y')}  dy'\, .
\end{equation}
The integral above is well defined at $\pi$ since $M(y') / \sin (y')$ is integrable at $\pi$, and it is well-defined at $0$ due to the cancellation of $E$ at $(0, 0)$. Now, assume that the right-hand side in \eqref{eq:formula_G} is different from zero as $y\to \pi$. We would then get that $f(y)\approx 1/M(y)$ as $y\to \pi,$ and therefore $| \cH \omega (0, y) | \gtrsim |y-\pi|^{-\text{Re}(\lambda_\star)}$ as $y \to \pi$. This would imply that $f ( y) \notin H^{1/2-\text{Re}(\lambda_\star )}_y(\TT)$, hence $f \notin H^{1-\text{Re}(\lambda_\star )} (\mathbb T^2)$ or, equivalently, $\omega_\star \notin H^{1 - \text{Re}(\lambda_\star )} (\mathbb T^2)$. 
On the other hand, if the right-hand side in \eqref{eq:formula_G} is zero at $\pi$ one can work with the expression \eqref{eqn:invert-transp}, subtracting $2\lambda_\star \Delta^{-1} \omega_\star (0, 0)= 2\int_0^\infty e^{-s\lambda_\star}  \Delta^{-1} \omega_\star(0,0) \, ds$
on both sides, to see that the expected behavior of $\cH \omega_\star (x,y)
-
2\Delta^{-1} \omega_\star (0, 0)$ along the $x$ axis also appears on neighboring integral curves

\begin{equation} \label{eq:f_formula}
\cH \omega_\star (x,y)
-
2\Delta^{-1} \omega_\star (0, 0) = 2\lambda_\star | \tan(y/2) |^{\lambda_\star} I(\omega_\star)  + o(|(x,y)|^{6/5} ),
\end{equation}
where we define 
\begin{equation}\label{def:Iomegastar}
I(\omega_\star)\coloneqq
    \int_{0}^\pi \frac{\Delta^{-1}\omega_\star (0, y') - \Delta^{-1} \omega_\star (0, 0)}{\sin (y')} \frac{1}{|\tan(y'/2)|^{\lambda_\star}}  dy' .
\end{equation}
One is left to showing that this integral is nonzero, which can for instance be easily verified numerically.
\end{remark}

\subsection{Instability in Navier-Stokes}\label{sec:NS}

\begin{figure}[htbp]
    \centering
    \begin{subfigure}[b]{0.48\textwidth}
        \centering
        \includegraphics[width=\textwidth]{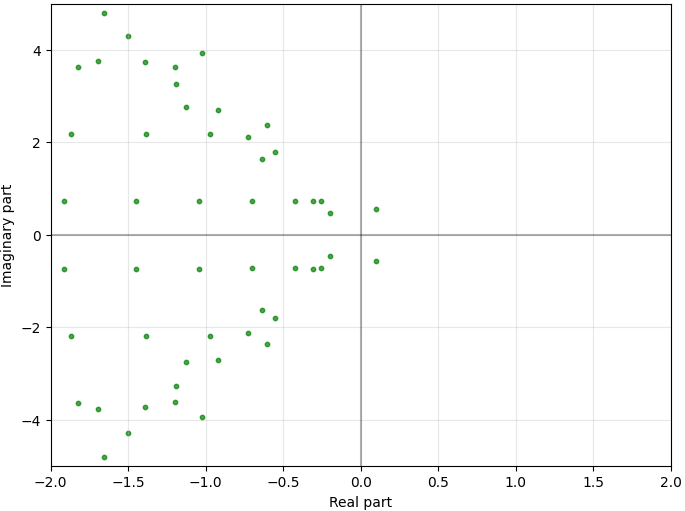}
        \label{fig:eigenfunction2}
    \end{subfigure}
    \hfill 
    \begin{subfigure}[b]{0.48\textwidth}
        \centering
        \includegraphics[width=\textwidth]{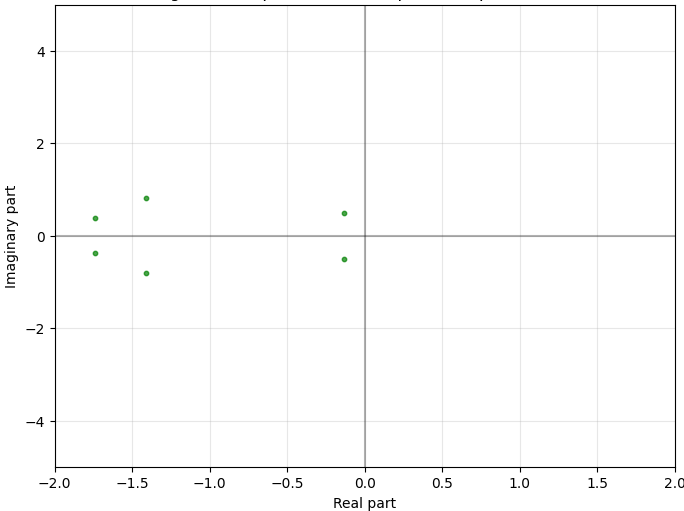}
        \label{fig:instability2}
    \end{subfigure}
    \caption{ Numerical spectrum of $\cL_\nu$. Left plot corresponds to $\nu = 0.01$ and right one to $\nu = 0.2$. We observe how a complex eigenpair $(\lambda_\nu, \bar \lambda_\nu)$ with $\Re (\lambda_{\nu} ) > 0$ persists for $\nu = 0.01$ but becomes stable at $\nu = 0.2.$
    }
    \label{fig:spectra_dissipation}
\end{figure}

We end this section with the proof of Corollary \ref{cor:NS}. 

We first observe that the space $(\rm{Ev}_x \rm{Ev}_y)^{[\rm{odd}]}_+$ in \eqref{EvEVpm} is invariant under $\cL_\nu=\cL_E + \nu \Delta$ for every $\nu$, as one can readily see by the invariance of $\Delta$ on each element of the basis of $(\rm{Ev}_x \rm{Ev}_y)^{[\rm{odd}]}_+$ in \eqref{def:Ejk}.

By means of the decomposition $\cL_E = \cJ(\cH_s + \cH_u)$ from \eqref{eq:lesdiablerets} and isolating the rank-$1$ perturbation $\cH_u$, we recast the eigenvalue problem $(\cL_\nu - \lambda)f = 0$ into
\begin{equation} \label{eq:viscous_split}
    (\cJ \cH_s + \nu \Delta - \lambda) f = - \cJ \cH_u f \,.
\end{equation}
As in the proof of Theorem \ref{thm:instability_criterion}, we express the eigenfunction equation as
\begin{equation} \label{eq:eigen_dissipation}
(\cT_\nu - \lambda) \cH_s^{\frac12} f = - \cT \cH_u (\cH_s^{-\frac12} f ) 
\end{equation}
where 
$\cT_\nu \coloneqq \cT + \nu \Delta$, with $ \cT = \cT_0 = \cH_s^{1/2} \cJ \cH_s^{1/2}$. 
To obtain \eqref{eq:eigen_dissipation} we also exploited that the operators $\Delta, \cH_s$ and $ \cH_u$ commute because they are all diagonal in the basis $\{ E_{j, k} \}_{(j, k) \in \cI}$ from \eqref{def:Ejk}.

The operator $\cT_\nu$ has only pure-point spectrum because it is a perturbation of the self-adjoint operator $\nu \Delta$ by the relatively compact operator $\cT$. If $\lambda$ is an eigenvalue of $\cT_\nu$, then, by classical energy estimates, we have $\Re(\lambda) \leq -\nu $. Consequently, for any $\lambda$ with $\Re(\lambda) > 0$, one obtains the uniform bound on the resolvent $\| (\cT_\nu - \lambda)^{-1} \|_{L^2\to L^2} \leq \frac{1}{\Re(\lambda)}$. 

As a consequence, we can take the resolvent in \eqref{eq:eigen_dissipation} for any $\lambda \in \mathbb C$ with $\mathrm{Re}(\lambda ) > 0$. One can proceed as in Theorem \ref{thm:instability_criterion} and deduce that there exist eigenvalues with $\mathrm{Re}(\lambda) > 0$ to $\cL_\nu$ if and only if the holomorphic function 
\begin{equation}
    \Phi_\nu(\lambda) := 1 - \frac{8}{3} \frac{1}{\| E_u \|_X^2} \langle (\cT_\nu -\lambda)^{-1} \cT E_u, E_u \rangle_{L^2}
\end{equation}
vanishes at $\lambda$, with $E_u = E_{1, 0} = \cos (x) - \im \cos(y)$.
We claim that $\lim_{\nu \to 0} \Phi_\nu  = \Phi$ pointwise and finish the proof assuming that claim. Then, since $\{\Phi_\nu\}$ is uniformly bounded on compact subsets of $\Re(\lambda)>0$, by Montel's convergence theorem it follows that $\lim_{\nu \to 0} \Phi_\nu = \Phi$ uniformly on every compact subset of $\{\mathrm{Re}(\lambda) > 0\}$. By Hurwitz's theorem, there will then be a family $\lambda_\nu$ of simple roots of $\Phi_\nu$, with $\lim_{\nu \to 0} \lambda_\nu = \lambda_\star$. 

Finally, let us prove the pointwise limit $\lim_{\nu \to 0} \Phi_\nu  = \Phi$. In order to show that, it suffices to show the weak convergence $f_\nu := (\cT + \nu \Delta - \lambda)^{-1} \cT E_u \rightharpoonup (\cT - \lambda)^{-1} \cT E_u =: f_0$ as $\nu \to 0^+$. We have that $\| f_\nu \| \leq \frac{\| \cT E_u \|}{\rm{Re}(\lambda )}$, so by Banach-Alaoglu, for every sequence $\{ \nu_n \}_n$, we have a subsequence $\nu_{n_i}$ such that $f_{\nu_i} \rightharpoonup f$. Therefore, for any test function $\phi$, we see
$$ \langle \cT E_u, \phi \rangle = \langle (\cT + \nu_{n_i} \Delta - \lambda ) f_{\nu_{n_i}}, \phi\rangle = \langle f_{\nu_{n_i}}, (-\cT + \nu_{n_i} \Delta - \bar \lambda) \phi \rangle \xrightarrow[i \to \infty]{} \langle f, (-\cT - \bar \lambda) \phi \rangle$$
We obtain that $(\cT - \lambda)f = \cT E_u$ and therefore $f = f_0$. Since every sequence $\{ \nu_n \}$ contains a subsequence along which $f_\nu$ converges weakly to $f_0$, we can conclude the desired weak convergence.

\begin{remark}[Orbital stability of the viscous operator in certain invariant spaces] In any $\cL_E$-invariant subspace ${\rm X}$ on which the operator $\cH$ has no negative directions, one readily obtains orbital stability for the linear system $\de_t w = \cL_\nu w$. Indeed, let $f\in {\rm X}$ be an eigenvector of $\cL_\nu$ associated with the eigenvalue $\lambda$. By mimicking the proof of Lemma \ref{lemma:uniqueness_algebra} and introducing the well-defined positive semidefinite self-adjoint operator $\cH^{\frac12}$,
we have
\begin{align*}
\Re(\lambda) \|\cH^{\frac12} f\|_{L^2}^2 &=
\Re(\lambda \langle f,\cH f\rangle) = \Re(\langle \cL_\nu f,\cH f\rangle)  \\
&=\Re( \langle \cJ \cH f,\cH  f\rangle  + \nu\langle \Delta  \cH^{\frac12}    f,\cH^{\frac12} f\rangle) \,=-\nu \|\nabla \cH^\frac12 f\|^2_{L^2},
\end{align*}
where, in the last step, we used that $\cJ$ is skew-adjoint and $\Delta$ and $\cH^{\frac12}$ commute. By exploiting the Poincaré inequality for $\TT^2$, we deduce that $\Re(\lambda) \| \cH^{\frac12}    f\|_{L^2}^2 \leq  - \nu   \| \cH^{\frac12}    f\|_{L^2}^2$
When $\cH f \neq 0$, this directly implies that $\Re(\lambda) \leq -\nu$. Vice versa, if $\cH f= 0$, then $\cL_Ef=0$ and $f$ is an eigenvector of $\Delta$, which has only negative integer eigenvalues. In conclusion $\sigma(\cL_\nu|_{\rm X}) \subset \big\{z \in \CC : \Re(z) \leq -\nu\big\}$, making the origin the global exponential attractor of the linear system.
\end{remark}

\section{Generalization to $(m,n)$ Taylor-Green Vortices}\label{sec:nmTG}
In this section, we consider the more general family of vortices:
\begin{equation}\label{eqn:psimn}
    \psi_{m, n}(x,y) \coloneqq -\sin(mx)\sin(ny) \,, \quad n,m \in \mathbb{N} \,.
\end{equation}
The corresponding vorticity is $\omega_{m, n} = \Delta \psi_{m, n} = (m^2+n^2)\psi_{m, n}$. Analogously to \eqref{cL}, the linearized operator around this equilibrium has the complex Hamiltonian structure 
\begin{equation} \label{eq:JH_mn}
    \cL_{m,n} = \cJ_{m,n} \cH_{m,n}\quad \text{with} \quad\cJ_{m,n} = -\{\psi_{m, n}, \cdot\} \,, \quad \cH_{m,n} = 1 + (m^2+n^2)\Delta^{-1} \,.
\end{equation}

It is clear that the unstable eigenvalues of Theorem \ref{linthm} produce linear instabilities of the $(n,n)$-vortex, for every $n$, by scaling, as the following remark details.

\begin{remark}[Scaling the unstable eigenvalues of the $(1,1)$ Taylor-Green vortex]\label{rem:dilation_TG} For any $n \in \mathbb{N}$, the dilation operator $Df (x, y) = f(nx, ny)$ has the properties 
\begin{equation*}
\cJ_{n, n} D = n^2 D \cJ_{1, 1}, \quad \mbox{ and } \quad \cH_{n, n} D = D \cH_{1, 1} \Rightarrow \cL_{n, n} D = n^2 D \cL_{1, 1}
\end{equation*}
In particular, if we now take $w \in L_0^2(\mathbb T)$ to be an eigenfunction of $\cL_{1, 1}$ with eigenvalue $\lambda$, we see
\begin{equation*}
\cL_{n, n} D w = n^2 D \cL w = n^2 \lambda D w,
\end{equation*}
so $Dw$ is an eigenfunction associated to the  eigenvalue $n^2 \lambda$ of $\cL_{n, n}$. Hence, Theorem \ref{linthm}, the spectrum of the linearized operator $\cL_{n, n}$ for the $(n, n)$ vortex satisfies:
\begin{equation}
    \{ n^2\lambda_\star, -n^2\lambda_\star, n^2\bar \lambda_\star, -n^2\bar \lambda_\star \} \subset \sigma (\cL_{n, n}) \,.
\end{equation}
\end{remark}

We expect, however, the general $(m,n)$-vortices have many more instabilities with respect to  Theorem \ref{linthm}. To this end, we first identify many invariant subspaces of $\cL_{m,n} $; to keep the analysis short, we then analyze one of them where we find a new type of real unstable eigenvalue with respect to the four complex ones of the Taylor-Green vortex in Section \ref{sec:nneqm}. Then, in Section~\ref{sec:2-2} we outline the arguments to fully characterize the spectrum of the $(2,2)$ case, and we present numerical evidence that the unstable spectrum does not contain only the scaled $(1,1)$ ones, but also real instabilities.

\subsection{Instability for $m\neq n$}
\label{sec:nneqm}
We first identify many invariant subspaces of $\cL_{m,n} $, given by
\begin{equation} \label{eq:nm_subdivision}
L_0^2(\mathbb T^2) = \bigoplus_{\kappa = 0}^{n-1} \bigoplus_{\kappa_2 = 0}^{m-1} \bigoplus_{\rm C \in \{ 0, 1 \} } L_0^2(\mathbb T^2)^{[\rm C]}_{[\kappa_1, \kappa_2]}
\end{equation}
where the subindices $\kappa_1$ (resp. $\kappa_2$) denote the congruence of the harmonics in $x$ (resp. $y$) modulo $n$ (resp. $m$). The term $\rm C$ denotes the ``checkerboard" invariance, namely, fixing our representatives so that $\kappa_1 \in \{ 0, 1, \ldots m-1 \}$ and $\kappa_2 \in \{ 0, 1, \ldots n-1 \}$, the parity of $\frac{j-\kappa_1}{m} + \frac{k - \kappa_2}{n}$ remains invariant. A basis is given by $\exp (i[ (j'm+\kappa_1)x + (k'n+\kappa_2) y ])$ with $j' +k' \equiv \rm C$ (mod 2). Moreover, for the case $\kappa_1 = \kappa_2 = \rm C = 0$, we naturally exclude the frequency $j' = k' = 0$ due to the zero-average constraint.
It is evident that those subspaces are invariant under $\cH_{m,n}$. In order to check that they are invariant under $\cJ_{m,n}$, we notice that:
\begin{align*}
&\cJ_{m,n} e^{\im  (jx + ky)} \\
&\quad=  \left( -\frac{e^{\im mx} - e^{-\im mx}}{2\im} \cdot \frac{e^{\im ny} + e^{-\im ny}}{2} \im j + \frac{e^{\im mx} + e^{-\im mx}}{2} \frac{e^{\im ny} - e^{-\im ny}}{2\im } \im k   \right)e^{\im (jx + ky)}
\end{align*}

We also notice that, in contrast with the decomposition in invariant spaces of the $(1, 1)$ Taylor-Green vortex, the spaces $L_0^2(\mathbb T^2)_{[\kappa_1, \kappa_2]}^{[C]}$ do not admit a further decomposition regarding the evenness or oddness in $x$ or $y$. The reason is that, unless $\kappa_1 = 0$ (resp. unless $\kappa_2 = 0$), the space does not contain odd or even nonzero functions in $x$ (resp. in $y$).
\color{black}

\begin{theorem} \label{thm:mn} Consider the $(m, n)$-Taylor-Green vortex with $m \leq n$ such that 
\begin{equation} \label{eq:condition}
4a - 1 < \sqrt{n^2/m^2}+1 < 4a+1
\end{equation}
for some $a \in \mathbb Z_+$. Then, $\cL_{m, n} : L_0^2(\mathbb T^2) \to L_0^2 (\mathbb T^2)$ has a real instability.
\end{theorem}
\begin{proof}
We consider the space $L_0^2(\TT^2)^{[1]}_{[0,0]}$. Crucially, since we choose $\kappa_1 = \kappa_2 = 0$, we notice that we can decompose our space as the direct sum of the subspaces of even and odd functions, both in $x$ and $y$. Clearly, the operator $\cL_{m,n}$ respects such parity, so the subspace $(\mathrm{Ev}_x \mathrm{Ev}_y)_{[0, 0]}^{[1]} \subset L_0^2(\mathbb T^2)_{[0, 0]}^{[1]}$, is invariant under $\cL_{m,n}$. A basis of this space is 
$$ \mathrm{B} = \{ \cos (j'm x) \cos (k' ny ) \; : \; j', k' \geq 0, \; \mbox{ and } \; j'+k' \equiv 1 \text{ (mod 2) }\}, $$
which clearly diagonalizes $\cH$. 

Our goal is to apply the instability criterion of {Corollary \ref{lem:weakinstability}}, of which we now verify the assumptions.

Regarding the negative directions of $\cH$, applying \eqref{eq:JH_mn} to our basis, we need to look for $j'$ and $k'$ such that $(mj')^2 + (nk')^2 < n^2+m^2$. Since $m \leq n$, negative directions must have $k'\leq 1$. If $k'= 1$, the only negative direction is $\cos (ny)$ ($(j', k') = (0, 1)$). When $k' = 0$, the $j'$ giving negative directions  are odd and such that $(j'm)^2 < m^2 + n^2$. Hence, the negative directions are
$$ \mathrm{N} = \{ \cos (ny) \} \cup \{ \cos (j'x) \; \mbox{ s.t. } \; j' \equiv 1 \mbox{ (mod 2)   and   } (j'm)^2 < m^2+n^2 \}. $$
We then observe that $|\mathrm{N}|$ is odd: indeed, there is an even amount of indices $j'$ in the definition of $\mathrm N$, since by assumption~\eqref{eq:condition} they coincide with the numbers $1, 3, \ldots 4a-1$.

{Next, we claim that the projection over $\ker \cJ$, namely the integral over the streamlines of the $(m,n)$ Taylor-Green vortex operator \eqref{eqn:projectionkerJ}, is constantly $0$ on the space $(\mathrm{Ev}_x \mathrm{Ev}_y)_{[0, 0]}^{[1]}$. By linearity, we reduce to prove that each element of  $\rm{B}$ has $0$ integral over all streamlines. For every element of  $\rm{B}$ either $j'$ or $k'$ is odd. In particular, since $\cos (j'm x)$ changes sign under the change of variables $(x,y) \to ({\pi}/m-x,y)$ while the streamlines are invariant by this transformation, we deduce that any function of the type $\cos (j' m x) f(y)$ with $j'$ odd, or analogously of the type $g(x) \cos (k'n y)$ with $k'$ odd, will project zero to the streamlines. 

We are therefore in the position to apply {Corollary \ref{lem:weakinstability}} to get a real instability. 
}
\end{proof}

\subsection{The (2, 2) Taylor-Green spectrum}\label{sec:2-2}
\begin{figure}[htbp]
    \centering
\includegraphics[width=0.7\textwidth]{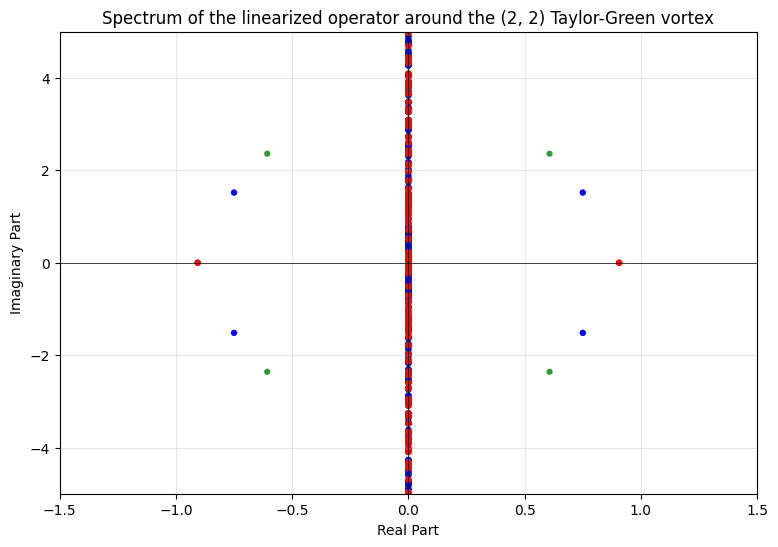}
    \caption{Numerical plot of the spectrum of $\cL_{2, 2}$. Green corresponds to eigenvalues on the invariant space $L_0^2(\mathbb T^2)_{[0, 0]}$, blue to $L_0^2(\mathbb T^2)_{[1, 0]}$ (or $L_0^2(\mathbb T^2)_{[0, 1]}$, which has the same spectrum) and red corresponds to $L_0^2(\mathbb T^2)_{[1, 1]}$. All have essential spectrum on $i \mathbb{R}$. The multiplicity of the eigenvalues on $\mathbb{C}\setminus i\mathbb{R}$ is $4$ for the red or blue ones, and $1$ for the green ones.}
    \label{fig:spectrum_L22}
\end{figure}
While Remark \ref{rem:dilation_TG} and Theorem \ref{thm:mn} just prove the existence of complex and real instabilities for certain $(m, n)$ Taylor-Green vortices, let us remark that our instability criterion is highly generalizable to other choices of $(m, n)$, and in particular, is also able to fully characterize the point spectrum $\sigma (\cL_{m, n}) \setminus i\mathbb R$. Since the divisions in subspaces and cases depend highly on the choice of $(m, n)$, let us illustrate this idea by focusing our discussion on the $(2, 2)$ Taylor-Green vortex. Rather than giving a rigorous proof of the characterization, we will discuss the space splitting and how one could perform the proof on each of those subspaces in a completely analogous way to the proofs we have presented for $\cL_{1, 1}$.

We start studying $\cL_{2, 2}$ by looking at each of the invariant spaces from \eqref{eq:nm_subdivision}. We have plotted numerically the spectrum of $\cL_{2, 2} : L_0^2 (\mathbb T^2)_{[\kappa_1, \kappa_2]} \to L_0^2 (\mathbb T^2)_{[\kappa_1, \kappa_2]}$ for $\kappa_1, \kappa_2 \in \{0, 1\}$ in Figure \ref{fig:spectrum_L22}. In the space $L_0^2(\mathbb T^2)_{[0, 0]}$, the reasoning from Remark \ref{rem:dilation_TG} tells us that the spectrum of $\cL_{2, 2}$ is a $4$-times rescaled version of the spectrum of $\cL_{1, 1}$ over $L_0^2(\mathbb T^2)$ (which was fully characterized in Theorem \ref{linthm}). Thus, we just need to discuss $\cL_{2, 2}$ on the invariant spaces $L_0^2(\mathbb T^2)_{[1, 0]}$ and $L_0^2(\mathbb T^2)_{[1, 1]}$ (since the spectrum on $L_0^2(\mathbb T^2)_{[0, 1]}$ is the same as on $L_0^2(\mathbb T^2)_{[1, 0]}$ by conjugating via the involution defined in \eqref{sigmaoperator}). 

Focusing on the subspaces $L^2_{0}(\mathbb T^2)_{[1, 0]}$ or $L^2_{0}(\mathbb T^2)_{[1, 1]}$, we observe that the involution $\varrho_x f(x, y) = f(-x, y)$ conjugates $\cL_{2, 2}$ in the space  $L_{0}(\mathbb T^2)_{\kappa}^{[0]}$ with  $L_{0}(\mathbb T^2)_{\kappa}^{[1]}$, for $\kappa \in \{ [1, 0], [1, 1] \}$, analogously to the reasoning in \eqref{sigmaoperator}--\eqref{eqn:studiamounosolo}. Hence it suffices to study the spectrum of $\cL_{2, 2}$ on $L_{0}(\mathbb T^2)_{[1, 0]}^{[0]}$ and $L_{0}(\mathbb T^2)_{[1, 1]}^{[0]}$.

\smallskip 
\textbf{Study of $L_0^2(\mathbb T^2)_{[1, 0]}^{[0]}$.} Here, we note that since $\kappa_2 = 0$, we can express the space as the direct sum of even functions in $y$ and odd functions in $y$. Each of those subspaces is clearly invariant under $\cL_{2, 2}$. 

$\rm{Ev}_y(\mathbb T^2)_{[1, 0]}^{[0]}$ has basis $e^{\im (2j'+1)x} \cos (2k'y)$, with $j'\in \mathbb Z$, $k' \geq 0$, $j'+k'$ even, which clearly diagonalizes $\cH_{2, 2}$. We see that $\cH_{2, 2}$ has two negative modes, namely: $e^{ix}$ and $e^{-ix} \cos (2y)$. In this subspace we can find the blue complex unstable eigenpair $(\lambda_\circ, \bar \lambda_\circ)$ from Figure \ref{fig:spectrum_L22} (the complex pair with smaller imaginary part) and a rigorous proof could be done combining our criterion in Theorem \ref{thm:instability_criterion} with a computer-assisted winding argument as in Lemma \ref{lemma:holo_root}. Since there are only two negative modes of $\cH_{2, 2}$, that would imply the rest of the spectrum is purely imaginary, completely characterizing the spectrum of $\cL_{2, 2}$ on $\rm{Ev}_y(\mathbb T^2)_{[1, 0]}^{[0]}$. The same eigenvalues appear by symmetry on $\rm{Ev}_y(\mathbb T^2)_{[1, 0]}^{[1]}$ or $\rm{Ev}_x(\mathbb T^2)_{[0, 1]}^{[C]}$ for $C \in \{0, 1\}$, and hence the eigenvalues $\lambda_\circ, \bar \lambda_\circ$ would appear with multiplicity $4$ on $\sigma (\cL_{2, 2})$.

Regarding $\rm{Odd}_y(\mathbb T^2)_{[1, 0]}^{[0]}$, a basis of this space diagonalizing $\cH_{2, 2}$ would be $e^{\im (2j'+1) x} \sin (2k'x)$ with $j'\in \mathbb Z$, $k' \geq 1$ and $j'+k'$ even. The only negative mode is $e^{-\im x} \sin (2y)$. Although this space is not invariant under standard conjugation, it does admit a real structure under the involution $f(x, y) \to \overline{ f(-y, x) }$, and with that real structure $\cL_{2, 2}$ is a real Hamiltonian operator. In particular, if there was an unstable eigenvalue, it would need to be real. Numerics suggest that no such real eigenvalue exist in this space. One could prove that analogously to Proposition \ref{prop:OddEvstability}, combining our instability criterion Theorem \ref{thm:instability_criterion}, the monotonicity formula for $\Phi |_{\mathbb R_+}$ (Proposition \ref{prop:propertiesofPhi}) and estimating $\Phi (0)$.

\smallskip
\textbf{Study of $L_0^2(\mathbb T^2)_{[1, 1]}^{[0]}$.} Even though this space does not contain odd or even nonzero functions in $x$ or $y$, it does contain functions with even and odd symmetry when reflected radially across the point $(0, 0)$. Indeed, if $e^{\im (jx + ky)} \in L_0^2(\mathbb T^2)_{[1, 1]}^{[0]}$, we also have that $e^{\im (-jx -ky)} \in L_0^2(\mathbb T^2)_{[1, 1]}^{[0]}$. We can thus split the space as a direct sum $L_0^2(\mathbb T^2)_{[1, 1]}^{[0], c} \oplus L_0^2(\mathbb T^2)_{[1, 1]}^{[0], s}$ where the first one has basis $\cos ( (2j'+1) x + (2k'+1) y)$ with $j'+k'$ even and $j' \geq 0$; and the second one has basis $\sin ( (2j'+1) x + (2k'+1) y)$ with $j'+k'$ even and $j' \geq 0$. One can reduce to study the spectrum on one of the two spaces by conjugating $\cL_{2, 2}$ with the operator
$\sigma_t f (x, y) = f( \pi / 4-x, \pi / 4-y)$, which sends one basis in the other up to signs. 

The only negative mode of $\cH_{2, 2}$ on $L_0^2(\mathbb T^2)_{[1, 1]}^{[0], c}$ is $\cos (x + y)$, implying that if there is an instability, it must be real. Numerically, we do see a real instability in this space (red eigenvalues in Figure \ref{fig:spectrum_L22}). This could be proven in an analogous way to Proposition \ref{prop:OddEvstability}, but in this case, since the real instability exists, we expect $\Phi (0) < 0$. That would  characterize the spectrum of $\cL_{2, 2} : L_0^2(\mathbb T^2)_{[1, 1]}^{[0], c} \to L_0^2(\mathbb T^2)_{[1, 1]}^{[0], c}$ as $\im \mathbb R \cup \{ \lambda_\triangle, -\lambda_\triangle \}$, with $\lambda_\triangle$ being a simple eigenvalue. It then becomes an eigenvalue of multiplicity $4$ of $\cL_{2, 2} : L_0^2(\mathbb T^2) \to L_0^2 (\mathbb T^2)$ since, due to symmetry, it also appears on $L_0^2(\mathbb T^2)_{[1, 1]}^{[0], s}$, $L_0^2(\mathbb T^2)_{[1, 1]}^{[1], c}$ and $L_0^2(\mathbb T^2)_{[1, 1]}^{[1], s}$.

\medskip

\noindent
{\bf Acknowledgements}
We thank Javier G\'omez-Serrano, Marcus Pasquariello and David Villringer for helpful discussions. 
M. Colombo,  G. Cao-Labora and P. Ventura were supported by the Swiss State Secretariat for Education, Research and Innovation (SERI) under contract number MB22.00034 through the project TENSE. G. Cao-Labora also acknowledges partial support from the MICINN research grant number PID2021–125021NA–I00. M. Dolce was supported  by the Swiss National Science Foundation (SNF Ambizione grant PZ00P2\_223294).

\bibliographystyle{siam}
\bibliography{bibTaylorGreen}

\end{document}